\documentclass[11pt,reqno]{amsart}
\usepackage{amssymb}
\usepackage{amsmath}
\usepackage{bm}
\usepackage{color,graphicx}
\usepackage{tikz}
\usetikzlibrary{decorations.pathreplacing}
\usepackage{times,mathptmx}
\usepackage{xspace}
\usepackage{longtable}
\usepackage[utf8]{inputenc}
\usepackage[centering,includehead,width=140mm,height=21cm]{geometry}
\usepackage[colorlinks=true]{hyperref}

\theoremstyle{plain}
\newtheorem{theorem}{Theorem}[section]
\newtheorem{Theorem}[theorem]{Theorem}

\newtheorem{Proposition}[theorem]{Proposition}

\newtheorem{Lemma}[theorem]{Lemma}

\newtheorem{Corollary}[theorem]{Corollary}
\theoremstyle{remark}
\newtheorem{Construction}[theorem]{Construction}
\newtheorem{Notation}[theorem]{Notation}

\newtheorem{Example}[theorem]{Example}

\newtheorem{Remark}[theorem]{Remark}

\newtheorem{alg}[theorem]{Algorithm}
\theoremstyle{definition}
\newtheorem{Definition}[theorem]{Definition}

\newcommand{\NN}{{\mathbb N}}

\newcommand{\PP}{{\mathbb P}}
\newcommand{\QQ}{{\mathbb Q}}

\newcommand{\RR}{{\mathbb R}}

\newcommand{\ZZ}{{\mathbb Z}}

\newcommand{\one}{{\mathbf 1}}

\newcommand{\Ct}{K\{\!\!\{t\}\!\!\}}

\DeclareMathOperator{\Newt}{Newt}
\DeclareMathOperator{\trop}{trop}
\DeclareMathOperator{\Trop}{Trop}

\DeclareMathOperator{\Supp}{Supp}

\DeclareMathOperator{\conv}{conv}
\DeclareMathOperator{\cone}{cone}

\DeclareMathOperator{\val}{val}
\DeclareMathOperator{\rec}{rec}
\DeclareMathOperator{\Rec}{Rec}
\DeclareMathOperator{\Subdiv}{Subdiv}

\DeclareMathOperator{\RHS}{RHS}

\newcommand{\ie}{i.\,e.\ }
\newcommand{\eg}{e.\,g.\ }
\newcommand{\cf}{cf.\ }

\renewcommand{\star}{\operatorname{star}}

\newlength {\sidetextwidth}
\newlength {\sidepicwidth}
\newsavebox {\sidepicbox}
\newenvironment {sidepic}[1]{%
  \par%
  \def\message##1{}%
  \hbadness=10000%
  \setlength{\sidetextwidth}{\textwidth}%
  \savebox{\sidepicbox}{\input{#1}}%
  \settowidth{\sidepicwidth}{\usebox{\sidepicbox}}%
  \addtolength{\sidetextwidth}{-6mm}%
  \addtolength{\sidetextwidth}{-\sidepicwidth}%
  \begin {minipage}{\sidetextwidth}%
  \parskip1ex plus0.5ex minus0.1ex}%
  {\end {minipage}%
  \hfuzz=10000pt%
  \hspace {4mm}%
  \begin {minipage}{\sidepicwidth}\usebox{\sidepicbox}\end {minipage}%
  \hfuzz=0pt%
  \hbadness=1000%
  }

\renewenvironment {enumerate}%
  {\begin {oldenumerate}\parskip1ex plus0.5ex \itemsep 0mm \parindent 0mm}%
  {\end {oldenumerate}}

\renewenvironment {itemize}%
  {\begin {olditemize}\parskip1ex plus0.5ex \itemsep 0mm \parindent 0mm}%
  {\end {olditemize}}

\sloppypar

\begin{document}
  \title[Realizability of Tropical Curves in a Plane in the Non-Constant Coefficient Case]{Relative Realizability of Tropical Curves in a Plane in the Non-Constant Coefficient Case}

\author{Anna Lena Birkmeyer and Andreas Gathmann}

\address{Anna Lena Birkmeyer, Fachbereich Mathematik, Technische Universität
  Kaiserslautern, Postfach 3049, 67653 Kaiserslautern, Germany}
  \email{birkmeyer@mathematik.uni-kl.de}

\address{Andreas Gathmann, Fachbereich Mathematik, Technische Universität
  Kaiserslautern, Postfach 3049, 67653 Kaiserslautern, Germany}
  \email{andreas@mathematik.uni-kl.de}


\begin{abstract}
  Let $X$ be a plane in a torus over an algebraically closed field $K$, with
  tropicalization the matroidal fan $\Sigma$. In this paper we present an
  algorithm which completely solves the question whether a given
  one-dimensional balanced polyhedral complex in $\Sigma$ is relatively
  realizable, \ie whether it is the tropicalization of an algebraic curve over
  $ \Ct $ in $X$. The algorithm implies that the space of all such relatively
  realizable curves of fixed degree is an abstract polyhedral set.

  In the case when $X$ is a general plane in $3$-space, we use the idea of this
  algorithm to prove some necessary and some sufficient conditions for relative
  realizability. For $1$-dimensional polyhedral complexes in $ \Sigma $ that
  have exactly one bounded edge, passing through the origin, these necessary
  and sufficient conditions coincide, so that they give a complete
  non-algorithmic solution of the relative realizability problem.
\end{abstract}

\maketitle

  \section{Introduction}

Although tropical geometry is growing in several directions over the last years
and a wide spectrum of algebraic concepts has already been transferred to the
tropical world, the tropicalization map, and in particular its image, is not
yet well understood. There is still no general approach to decide whether a
given combinatorial object comes from an algebraic one via the tropicalization
map. Results in this direction are currently restricted to rational curves and
hypersurfaces in $ \RR^n $ \cite{MI,SP,NS06}, as well as some partial
statements for the case of elliptic curves \cite{Spe07}.

The relative case of this realizability question is even less studied: if $K$
is an algebraically closed field, $X$ a variety in an $n$-dimensional torus
over the field $ \Ct $ of Puiseux series over $K$, $ \Sigma := \Trop X \subset
\RR^n $ its tropicalization, and $C$ a weighted polyhedral complex in $ \Sigma
$, is there an algebraic subvariety of $X$ that tropicalizes to $C$? So far
this question has only been considered in greater detail in the case when $X$
is a plane defined over $K$ and $C$ a $1$-dimensional balanced fan (the
so-called constant coefficient case). There is then an algorithm to determine
which fans are tropicalizations of curves, as well as some necessary and some
sufficient conditions for relative realizability
\cite{bogartkatz,brugalleshaw,wir}.

In this paper, we will study this relative realizability problem in the case
when $X$ is still a plane defined over $K$, hence $ \Sigma $ a $2$-dimensional
matroidal fan, but $C$ is now a $1$-dimensional balanced polyhedral complex
(which is not necessarily a fan). An example of this situation is shown in the
picture below on the left. After introducing the required background from
polyhedral and tropical geometry in Section \ref{ch:tropgeo}, we will present
an algorithm to decide if $C$ is relatively realizable in Section
\ref{sec-realize}. As in the constant coefficient case in \cite{wir}, this
approach is based on projections to coordinate planes, so that one can work
with tropical curves in $ \RR^2 $ and their dual picture of extended Newton
polytopes. We show that it is sufficient to consider Puiseux series in $ \Ct $
containing only a fixed finite set of powers of $t$, so that the calculations,
presented in Algorithm \ref{alg-1}, can be performed on a computer. The
algorithm is implemented as a library for the computer algebra system Singular,
as explained in Remark \ref{rem-singular} \cite{Sing,Win12}. With its help, we
show in Proposition \ref{prop-polyhedral} that the space of realizable tropical
curves is a polyhedral set in the moduli space of all tropical curves in $
\Sigma $ of fixed degree.

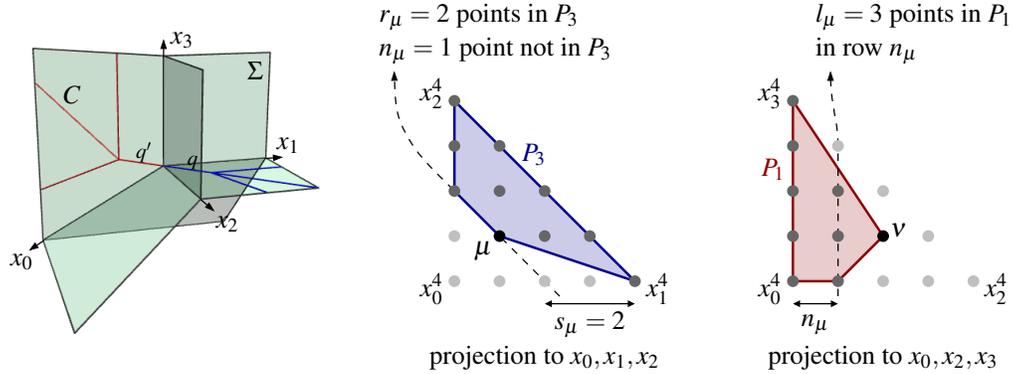
\def\message#intro{}%
  {\begin{center}\input{intro}\end{center}}

The main part of the final Section \ref{sec-criteria} then deals with the
special case when $X$ is a general plane in $3$-space, so that $ \Sigma $ is
the matroidal fan in the picture above. We then prove some necessary and some
sufficient non-algorithmic conditions for relative realizability, both in the
constant and the non-constant coefficient case. They are particularly strong in
the special situation of the picture when $C$ has exactly one bounded edge,
passing through the origin with homogeneous direction $ [0,1,1,0] $ and lengths
$ q,q' \in \RR_{\ge 0} $ on the two sides, and hence with unbounded ends
contained in two opposite cones of $ \Sigma $. In this case the necessary and
sufficient conditions agree, thus giving a complete non-algorithmic solution of
the relative realizability problem.

Let us quickly describe these conditions. The curve $C$ is described completely
by the two lengths $q$ and $q'$, together with the Newton polytopes $ P_3 $ and
$ P_1 $ of the projections of $C$ to the $ x_0,x_1,x_2 $-plane and the $
x_0,x_2,x_3 $-plane, respectively, as shown in the picture above on the right.
By a row of these polytopes we will mean a line parallel to the
$x_1$-$x_2$-side in the case of $ P_3 $, and to the $x_0$-$x_3$-side in the
case of $ P_1 $. We denote by $ \Delta $ the standard simplex of size $ d :=
\deg C $, indicated by the dots in the picture.

Now for any vertex $ \mu $ of one of the polytopes $ P_3 $ or $ P_1 $, consider
the lattice points of $ \Delta $ in the row of $ \mu $. Let $ r_\mu $ be the
number of these points contained in the polytope, and $ n_\mu $ the number of
these points not in the polytope, so that $ s_\mu := d + 1 - r_\mu - n_\mu $ is
the row number of $\mu$. Moreover, let $ l_\mu $ be the number of lattice
points in row $ n_\mu $ of the other of the two polytopes. The condition on $
\mu $ is then that $ l_\mu \ge r_\mu $. We show in Proposition
\ref{prop:posreal} that the recession fan of $C$ (\ie the corresponding fan
curve if $ q=q'=0 $) is relatively realizable if and only if this condition is
satisfied for all vertices of $ P_3 $ and $ P_1 $. It is easily checked that
this is the case for the curve in the picture above.

Moreover, if the recession fan of $C$ is relatively realizable, Proposition 
\ref{prop:posreal2} states that the (non-constant coefficient) curve $C$ is 
relatively realizable if and only if for any vertex $\mu$ of $P_3$ with 
row number $s_\mu \neq 0$ and any vertex $\nu$ of $P_1$ with $n_\nu \neq 0$, 
the lengths $q$ and $q'$ satisfy the condition
  \[ \frac{n_\mu}{s_\mu} \le \frac q{q'} \le \frac{s_\nu}{n_\nu}. \]
For example, for the two vertices $ \mu $ and $ \nu $ in the picture above we
have $ n_\mu = 1$ and $ s_\mu = n_\nu = s_\nu = 2 $, and so we know that
$C$ can only be relatively realizable if $ \frac 12 \le \frac q{q'} \le 1 $.
The other two relevant vertices of the polytopes give us the same lower and
upper bound for $\frac q{q'}$, hence we conclude that $C$ is relatively
realizable if and only if $ \frac 12 \le \frac q{q'} \le 1 $.

It may happen that there are vertices $\mu$ of $P_3$ and $\nu$ of $P_1$ such
that $\frac{n_\mu}{s_\mu} > \frac{s_\nu}{n_\nu}$ and so, the tropical curve $C$
is not relatively realizable for any $q,q' \in \mathbb R_{>0}$. One may ask if
this is only possible if the recession fan of $C$ is not realizable. However,
in Examples \ref{ex-rec} and \ref{ex-empty} we show that the solvability of the
length conditions does not correlate with the relative realizability of the
recession fan of $C$.

 % Introduction
  \section{Preliminaries}\label{ch:tropgeo}

As a preparation for the main part of this paper, we need to introduce some
basic concepts and statements in tropical geometry to explain our setup. The
following section is concerned with polyhedral theory and the way we want to
tropicalize algebraic varieties. We start with purely tropical objects, and
then explain how to obtain them from algebraic ones.

\subsection{Tropical varieties}

\begin{Definition}[Weighted polyhedral complex]
 Let $n \in \mathbb N$, let $\Lambda$ be a lattice of rank $n$, and let $V =
 \Lambda \otimes_{\mathbb Z} \mathbb R$ be the corresponding real vector space.
 A \emph{(rational) weighted polyhedral complex} $(P,\omega)$ in $V$ consists
 of a rational, pure-dimensional polyhedral complex $P$ in $V$ of dimension $m$
 together with a weight function $\omega$ on the maximal cells of $P$, \ie a
 function $\omega: \{ \sigma \in P : \dim(\sigma) = m \} \to \mathbb N$. We
 will often just write $P$ instead of $(P,\omega)$ if the weight function is
 clear from the context.

 If $ U \subset V $ is a $ \Lambda $-rational linear subspace of $V$, the
 quotient $ V/U $ will always be equipped with the induced lattice. The vector
 space $ \RR^n $ will always be considered with the underlying lattice $ \ZZ^n
 $.
\end{Definition}

The following definitions follow \cite{Rau09}.

\begin{Definition}[Star of a polyhedral complex]\label{def:star}
 Let $(P,\omega)$ be a rational weighted polyhedral complex in $ V = \Lambda
 \otimes_\ZZ \RR $, and let $\tau \in P$ be a polyhedron. Let $V_\tau$ be the
 linear subspace generated by $\tau$ and consider the quotient map $q: V \to
 V / V_\tau$. For $\sigma \in P$, denote by $\overline \sigma$ the cone spanned
 by the image of $\sigma - \tau$ under $q$.
 Then the \emph{star of} $P$ at $\tau$ is defined to be the fan
 $$ \star_P(\tau) := \{ \overline \sigma : \sigma \in P, \tau < \sigma \}. $$
 By setting $\omega'(\overline \sigma) = \omega(\sigma)$, we get a weighted fan $(\star_P(\tau),\omega')$.
\end{Definition}

\begin{Definition}[Tropical varieties] \label{def-tropvar} \text{}
  \begin{enumerate}
  \item \label{def-tropvar:a}
    Let $(C,\omega)$ be a rational weighted one-dimensional polyhedral fan in
    $ V = \Lambda \otimes_\ZZ \RR $. We assume that $\{0\}$ is a cone of $C$.
    For any one-dimensional cone $\sigma \in C$, let $v_\sigma$ be a primitive
    generator of $ \sigma \cap \Lambda $. Then we say that $(C,\omega)$ is
    \emph{balanced} if
      $$ \sum_{\substack{\sigma \in C \\[0.5ex]
         \sigma \text{ one-dimensional}}}
         \omega(\sigma) \, v_\sigma =0. $$ 
  \item \label{def-tropvar:b}
    A \emph{tropical variety} $(P,\omega)$ is a weighted polyhedral complex
    such that $ \star_P(\tau) $ is balanced as in \ref{def-tropvar:a} for all
    codimension-1 cells $\tau$ of $P$. The \emph{dimension} of a tropical
    variety is the dimension of its polyhedral complex; a \emph{tropical curve}
    is a tropical variety of dimension one.

    We want to identify two tropical varieties $(P,\omega)$ and $(P',\omega')$
    if their polyhedral complexes have the same support, and there is a common
    refinement of $\Supp(P)$ such that the induced weights of $\omega$ and
    $\omega'$ coincide.
  \end{enumerate}
\end{Definition}

The following definition follows \cite[Definition 1.11]{Zie}.

\begin{Definition}[Recession fan of a tropical variety] \label{def-recession}
 Let $(P,\omega)$ be a tropical variety in $ V = \Lambda \otimes_\ZZ \RR $, and
 let $\sigma \in P$ be a polyhedron. The \emph{recession cone} of $\sigma$ is
 defined as 
   $$ \rec(\sigma) = \{ y \in V: x+ty \in \sigma \text{ for all }
      x \in \sigma \text{ and } t \geq 0 \}. $$
 We define the \emph{recession fan} of $P$ as
   $$ \Rec(P) = \{ \rec(\sigma): \sigma \in P\}. $$
 \cite[Lemma 1.4.10]{Rau09} ensures that there is a polyhedral structure $P'$
 on the support of $P$ such that $\Rec(P')$ is a fan. Since we identify two
 tropical varieties if the polyhedral structures have a common refinement
 respecting the weights, we assume from now on that the polyhedral structure
 $P$ is chosen in such a way that $\Rec(P)$ is a fan.

 Moreover, we define weights on the maximal cones of $\Rec(P)$ by setting 
   $$ \omega'(\sigma) = \sum_{\substack{\sigma' \in P \\[0.5ex]
      \rec(\sigma') = \sigma}}
      \omega(\sigma').$$
 Then the weighted fan $(\Rec(P),\omega')$ is a tropical variety
 \cite[Definition 1.4.11]{Rau09}.
\end{Definition}

\begin{Definition}[Degree of a tropical curve]\label{def:degree}
 Let $C$ be a tropical curve in $\mathbb R^n$. We define the \emph{degree} of
 $C$ to be the degree of its recession fan $\Rec(C)$ in the sense of
 \cite[Definition 2.8]{wir}, \ie the intersection product of $C$ (or $ \Rec(C)
 $) with a general tropical hyperplane.
\end{Definition}

\begin{Example} \label{ex-1}
  We consider the one-dimensional polyhedral complex $P$ in $ \RR^2 $ pictured
  below with weight $1$ on all maximal cells. Its star at $ \tau $ is the fan
  with the three edges $ \cone (1,0) $, $ \cone(0,1) $, and $ \cone(-1,-1) $,
  again with weight $1$ on all maximal cells. As the primitive generators of
  these cones add up to $0$ (and similarly for $ \tau' $), we see that $P$ is a
  tropical curve. It has degree $2$ and recession fan consisting of the four
  edges $ \cone(\pm 1,0) $ and $ \cone(0,\pm 1) $.

  \def\message#ex-1{}%
  {\begin{center}\input{ex-1}\end{center}}
\end{Example}

\subsection{Tropicalization}

Let $K$ be an algebraically closed field (of any characteristic).

\begin{Definition}[Puiseux series] \label{def-puiseux}
  By $ \Ct $ we denote the field of \emph{generalized Puiseux series} over $K$
  in the sense that its elements are formal series of the form $ \sum_{k \in
  \RR} a_k t^k $ in a formal variable $t$ such that $ \{ k \in \RR: a_k \neq 0
  \} $ is a well-ordered subset of $ \RR $.
\end{Definition}

There are several different but equivalent ways to define the tropicalization
of an algebraic variety over $ \Ct $. We choose the following one because we
will use it in our algorithm.

For $n \in \mathbb N$, we will often use the notation $\Ct[x]$ instead of
$\Ct[x_0,\ldots,x_n]$, and for $\nu \in \mathbb N^{n+1}$, we also write $x^\nu$
instead of $x_0^{\nu_0} \cdots x_n^{\nu_n}$. We denote by $ e_i $ for $
i=0,\dots,n $ the $i$-th unit vector in $ \RR^{n+1} $, and set $ \one :=
\sum_{i=0}^n e_i $.

\begin{Definition}[Tropical hypersurface]
 Let $f = \sum_\nu a_\nu x^\nu \in \Ct[x_0,\ldots,x_n]$ be a homogeneous
 polynomial. The \emph{tropical polynomial associated to} $f$ is defined as the
 map
   \[ \trop(f): \mathbb R^{n+1} \to \mathbb R, \;\;
      y \mapsto \min_\nu ( \val(a_\nu) + y \cdot \nu). \]
 The \emph{tropical hypersurface} $\Trop(f)$ of $f$ is the locus of
 non-differentiability of $\trop(f)$, \ie
   $$ \Trop(f) = \{ y \in \mathbb R^{n+1}/\RR \cdot \one:
      \text{ the minimum in } \trop(f)(y)
      \text{ is achieved at least twice} \}. $$
\end{Definition}

\begin{Definition}[Tropicalization]
  Let $ I \subset \Ct[x] $ be a homogeneous ideal. As a set, we define the
  \emph{tropicalization} of $I$ as
  $$ \Trop(I)
     = \bigcap_{\substack{f \in I \\ \text{$f$ homogeneous}}} \Trop(f)
     \quad \subset \RR^{n+1} / \RR \cdot \one. $$
 Using initial ideals, one can give $ \Trop(I) $ the structure of a polyhedral
 complex and put weights on its maximal cells in such a way that it becomes a
 tropical variety \cite[Theorem 3.2.5 and Definition 3.4.3]{MS}. If $I$ is the
 homogeneous ideal of a projective variety $ X \subset \PP^n_{\Ct} $, we write
 $ \Trop(I) $ also as $ \Trop(X) $. When drawing tropical varieties, we will
 always identify $ \RR^{n+1} / \RR \cdot \one $ with $ \RR^n $ by the map $
 [x_0,x_1,\dots,x_n] \mapsto (x_1-x_0,\dots,x_n-x_0) $.
\end{Definition}

\begin{Example} \label{ex-6}
  We want to compute $\Trop(I)$, where $I = (x_1^2+x_2^2-t^2 \, x_0^2) \subset
  \Ct[x_0,x_1,x_2]$. For $f = x_1^2+x_2^2-t^2 \, x_0^2 $ we have $\trop(f)(y) =
  \min(2y_0+2,2y_1,2y_2)$, and so

  \vspace{-1ex}

  \begin{sidepic}{ex-6}
    $$ \Trop(f) = \{ [0,y_1,y_1] : y_1 \leq 1 \} \cup
      \{ [0,y_1,1] : y_1 \geq 1 \} \cup \{ [0,1,y_2] : y_2 \geq 1\}. $$

   \vspace{0ex}

   Moreover, as $I = (f)$ is principal we have $\Trop(I) = \Trop(f)$:
   if $g = f \cdot h \in I$ is homogeneous, then $\trop(g)(y) = \trop(f)(y) +
   \trop(h)(y)$, so if the minimum in $\trop(f)(y)$ is achieved at least twice,
   then so is the minimum in $\trop(g)(y)$. Hence, we have $\Trop(f) \subset
   \Trop(g)$ and thus $\Trop(I) = \Trop(f)$. Using initial ideals, it can be
   seen that the weights on all cones of $ \Trop(I) $ are $2$.
 \end{sidepic}
\end{Example}

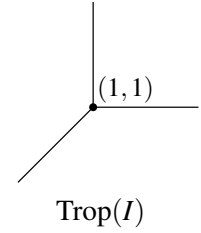
\begin{Example}[Tropicalization of linear spaces]
  Let $ X \subset \PP^n_{\Ct} $ be a linear space defined over $K$, so
  that its ideal $ L:=I(X) \subset \Ct[x_0,\dots,x_n] $ is generated by
  homogeneous linear polynomials over $K$. Assume that $X$ is not contained in
  any coordinate hyperplane, \ie that $L$ does not contain any monomial. Then
  the tropicalization of $X$ can be described in terms of matroid theory
  \cite{Oxl}: if $ M(L) $ is the matroid defined by $L$ then $ \Trop(X) $ is
  the Bergman fan of $ M(L) $ \cite[Theorem 4.1.11]{MS}, defined as follows.
\end{Example}

\begin{Construction}[Bergman fan of a matroid] \label{constr-bergman}
 Let $M$ be a matroid on the ground set $ \{0,\ldots,n\} $ for some $n \in
 \mathbb N$. For a flat $F$ of $M$, we define the associated vector 
 $v_F = \sum_{i\in F} [e_i] \in \mathbb R^{n+1} / \mathbb R \cdot \mathbf 1$.
 If $\mathcal F$ is a chain $F_0 \subset F_1 \subset \cdots \subset F_k$ of
 flats in $M$, let
 $\sigma_{\mathcal F}$ be the cone generated by all the associated vectors of flats in $\mathcal F$, \ie
 $\sigma_{\mathcal F} = \cone( v_{F_i} : i=0,\dots,k) $.

 These cones $ \sigma_{\mathcal F} $ form a fan in $ \RR^{n+1} / \RR \cdot
 \textbf 1 $, the so-called \emph{Bergman fan} $B(M)$ of $M$. Together with the
 constant weight function $1$, it is a tropical variety.
\end{Construction}

We are now ready to describe the main question studied in this paper. Given an
algebraic plane $ X \subset \PP^n_{\Ct} $ defined over $K$, so that its
tropicalization is a $2$-dimensional fan as above, we want to know whether a
given tropical curve in $ \Trop(X) $ can be realized as the tropicalization of
an algebraic curve in $X$. So roughly speaking, we are considering non-constant
coefficient tropical curves in a constant coefficient tropical plane. More
precisely, we will study the relative realizability of plane curves in the
following sense.

\begin{Definition}[Relative realizability of curves] \label{def-realize}
  Let $ X \subset \PP^n_{\Ct} $ be a plane defined over $K$ and not contained
  in any coordinate hyperplane, with homogeneous ideal $ L:=I(X) \subset \Ct[x]
  $ and tropicalization $ \Sigma := B(M(L)) \subset \RR^{n+1} / \RR \cdot \one
  $ as in Construction \ref{constr-bergman}. We say that a tropical curve $C$
  in $ \Sigma $ is \emph{(relatively) realizable} in $X$ (or $L$) if there is a
  homogeneous polynomial $ f \in \Ct[x] $ with $ \Trop (L+(f)) = C $ (including
  the weights). If $X$ is clear from the context, we also just say that $C$ is
  (relatively) realizable.
\end{Definition}

 % Preliminaries
  \section{Relative Realizability} \label{sec-realize}

As above, let $K$ be any algebraically closed field, and let $ X \subset
\PP^n_{\Ct} $ be an algebraic plane defined over $K$ which is not contained in
any coordinate hyperplane, with homogeneous ideal $ L \subset
\Ct[x_0,\dots,x_n] $. In this section, we will explain how projections can be
used to decide whether or not a tropical curve in $ \Sigma = \Trop(X) =B(M(L))
$ is relatively realizable in the sense of Definition \ref{def-realize}. These
projections $ \Sigma \to \RR^3 / \RR \cdot \one $ are obtained using matroid
theory, and allow us to apply already known statements for tropical curves in
$\mathbb R^2$.

\subsection{Projecting to the plane}\label{sec:proj}

After introducing the necessary tropical and algebraic projections, we will see
that they commute with tropicalization, and that a tropical curve in $ \Sigma $
can be reconstructed from all its images under these projections. This will
help us to develop the algorithm mentioned above.

\begin{Definition}[Tropical projection] \label{def-proj-trop}
 For a basis $ A=(j_0,j_1,j_2) $ of $M(L)$, with $ 0 \le j_0,j_1,j_2 \le n $,
 we denote by $p^A: \mathbb R^{n+1} / \RR \cdot \one \to \mathbb R^A / \RR
 \cdot \one $ the projection onto the coordinates of $A$.
\end{Definition}

\begin{Remark} \label{rem-proj-trop}
  By construction, it is obvious that tropical projections commute with
  taking recession fans as in Definition \ref{def-recession}. Moreover, we have
  $ \deg (p_*^A D) = \deg D $ for any tropical fan curve $D$ in $ \Sigma $ by
  \cite[Lemma 3.9]{wir}. Hence
    \[ \deg p_*^A C = \deg \Rec p_*^A C = \deg p_*^A \Rec C = \deg \Rec C =
       \deg C \]
  by Definition \ref{def:degree}.
\end{Remark}

\begin{Definition}[Algebraic projection]\label{def:algproj}
 Let $A = (j_0,j_1,j_2) $ be a basis of $M(L)$. Moreover, let $R =
 \Ct[x_0,\dots,x_n]$ and $R^A = \Ct[x_{j_0},x_{j_1},x_{j_2}] $. Then the map
 $R^A \to R/L$ given by $x_i \mapsto x_i + L$ is an isomorphism
 since $A$ is a basis of $M(L)$. Geometrically, it describes the projection
 isomorphism from $X$ to the plane $ \PP^2_{\Ct} $ with homogeneous coordinates
 $ x_{j_0} $, $ x_{j_1} $, $ x_{j_2} $.

 The ideal $L$ is generated by polynomials in $K[x]$, so there are unique
 $c_{i,j} \in K$ such that the inverse map of this isomorphism is given by
   $$ \pi^A: R/L \to R^A, \;\;
      \overline{x_i} \mapsto c_{i,j_0}x_{j_0} + c_{i,j_1}x_{j_1} + c_{i,j_2}x_{j_2}. $$
 For a polynomial $ f \in R $, we denote the polynomial $ \pi^A(\overline f)
 \in R^A $ by $ f_A $. By construction, it has the property that the zero cycle
 of $f$ in $X$ maps to the zero cycle of $ f_A $ under the projection
 isomorphism $ X \to \PP^2_{\Ct} $ defined by $A$ as described above.
\end{Definition}

\begin{Theorem} \label{thm:gubler}
 As above, let $ X \subset \PP^n_{\Ct} $ be a plane defined over $K$, with
 defining ideal $L$, and let $A$ be a basis of $M(L)$. Then
 $$ \Trop(f_A) = p_*^A (\Trop (L+(f))) $$
 for every homogeneous polynomial $ f \in \Ct[x] $.
\end{Theorem}

\begin{proof}
  As push-forwards commute with tropicalization by \cite[Theorem
  13.17]{gubler}, the statement follows since the zero cycle of $ L+(f) $ maps
  to the zero cycle of $ f_A $ under the projection determined by $A$.
\end{proof}

\begin{Theorem} \label{mainthm}
 Let $C \subset \Sigma $ be a tropical curve, and let $f \in \Ct[x]$ be a
 homogeneous polynomial. With notations as above, the following are equivalent:
 \begin{enumerate}
  \item $\Trop(L+(f)) = C$,
  \item for all bases $A$ of $M(L)$ we have $\Trop(f_A) = p_*^A C$.
 \end{enumerate}
\end{Theorem}

\begin{proof}
 Analogously to the proof of \cite[Corollary 3.6]{wir}, \cf \cite[Remark 3.8]{wir}, we can see that
 if $\Trop(f_A) = p_*^A C$ for all bases $A$ of $M(L)$, then $\Trop(L+(f)) = C$. 
 Using Theorem \ref{thm:gubler}, we also know that if $\Trop(L+(f)) = C$, then 
 $\Trop(f_A) = p_*^A C$ for all bases $A$ of $M(L)$.
\end{proof}

\subsection{Plane tropical curves}

We now want to collect some facts about tropical curves in $ \RR^3 / \RR \cdot
\one \cong \RR^2 $ which will help us with the computations. For simplicity, in
this section we will consider tropical curves in $\mathbb R^2$, \ie in the
non-homogeneous case. In our algorithm, where we use matroid theory, the
introduced concepts and facts will be transferred to the homogeneous setting,
identifying $\mathbb R^2$ with $\mathbb R^3 / \mathbb R \cdot \mathbf 1$ by
$(y_1,y_2) \mapsto [0,y_1,y_2]$.

Any tropical curve $ C \subset \RR^2 $ is realizable and thus the
tropicalization of a single polynomial, see for instance \cite[Proposition
2.4]{MI}. The combinatorial information about $C$ is encoded in the Newton
subdivision of this polynomial. We will therefore now present the concept of
Newton subdivisions. 

\begin{Definition}[Newton subdivisions] \text{}
  \begin{enumerate}
  \item Let $f = \sum_{|\nu| \leq d} a_\nu x^\nu \in \Ct[x_1,x_2]$ be a
    polynomial of degree $d$. We call $\Newt(f) = \conv( \nu: a_\nu \neq 0)$
    the \emph{Newton polytope} of $f$.

    If we consider the \emph{extended Newton polytope} $\conv( (\nu,y) : a_\nu
    \neq 0, y \geq \val(a_\nu) )$ in $\mathbb R^2 \times \mathbb R$, then the
    projection $\mathbb R^2 \times \mathbb R \to \mathbb R^2$ forgetting the
    last coordinate induces a subdivision of $\Newt(f)$ analogous to
    \cite[Example 7.2.2]{gelf}. This subdivision of $\Newt(f)$ is called the 
    \emph{Newton subdivision} of $f$.
  \item Let $C \subset \mathbb R^2$ be a tropical curve. The \emph{Newton
    polytope} of $C$, denoted by $\Newt(C)$, is defined as the Newton polytope
    of the recession fan $\Rec(C)$ of $C$, as defined for example in
    \cite[Definition 4.13]{wir}. It is the unique polytope such that $\Rec(C)$
    is its inner normal fan and there is $d \in \mathbb N$ such that $\Newt(C)$
    touches all three faces of the standard $d$-simplex. It has a \emph{Newton
    subdivision} dual to $C$.
  \end{enumerate}
\end{Definition}

\begin{Proposition} \label{prop:dualsubdivision}
 Let $f \in \Ct [x_1,x_2]$ be a polynomial. Then the tropicalization $\Trop(f)$ is dual to the Newton
 subdivision of $f$. Moreover, the weights on the maximal cells are given by the lattice
 lengths of the corresponding faces in the Newton subdivision of $f$.
\end{Proposition}

\begin{proof}
The duality of the Newton subdivision of $f$ and its tropicalization is shown for instance in 
\cite[Proposition 3.1.6]{MS}. The equality of the weights and the lattice lengths can be found in
\cite[Proposition 3.4.6]{MS}.
\end{proof}

\begin{Remark} \label{rem:condsub}
  In the proof of \cite[Proposition 3.1.6]{MS}, one can see that if $\tau =
  \conv(\nu,\eta)$ is a $1$-dimensional face in the Newton subdivision of $f$,
  then the cell of $\Trop(f)$ dual to $\tau$ consists of all $y \in \mathbb
  R^2$ such that
    $$ \val(a_\nu) + y \cdot \nu = \val(a_\eta) + y \cdot \eta
       \leq \val(a_\mu) + y \cdot \mu $$
  for all $\mu \in \Newt(f) \cap \mathbb N^2$, with the convention that $
  \val(a_\mu) = \infty $ if $ a_\mu = 0 $ (so that we do not get a condition
  from such a lattice point).

  Hence, if $y$ is the vertex of $\Trop(f)$ corresponding to a $2$-dimensional
  polytope $Q$ in the Newton subdivision, we have
    $$ \val(a_\nu) + y \cdot \nu = \val(a_\mu) + y \cdot \mu$$
  for all vertices $\mu,\nu$ of $Q$, and
    $$\val (a_\nu) + y \cdot \nu \leq \val(a_\mu) + y \cdot \mu$$
  for all vertices $\nu$ of $Q$ and $\mu \in \Newt(f) \cap \mathbb N^2$. Note
  that these equations suffice to determine $y$ uniquely.
\end{Remark}

\begin{Example} \label{ex:newton}
 We consider the polynomial $f = t^2 + tx_1 -2t x_2 + tx_1^2 - 2tx_2^2 + x_1x_2$ in $\Ct[x_1,x_2]$. 
 Its Newton subdivision and tropicalization are as follows.
 
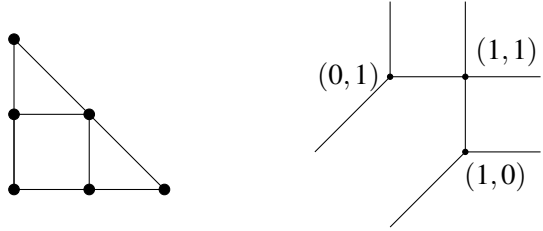
\begin{center}
  \begin{tikzpicture}
   \draw (0,-0.5) -- (1,-0.5) -- (1,0.5) -- (0,0.5) -- (0,-0.5);
   \draw (1,-0.5) -- (2,-0.5) -- (0,1.5) -- (0,-0.5);
   \filldraw (0,-0.5) circle (2pt);
   \filldraw (1,-0.5) circle (2pt);
   \filldraw (2,-0.5) circle (2pt);
   \filldraw (0,0.5) circle (2pt);
   \filldraw (0,1.5) circle (2pt);
   \filldraw (1,0.5) circle (2pt);   

   \draw (5,2) -- (5,1) -- (4,0);
   \draw (5,1) -- (7,1);
   \draw (6,2) -- (6,0) -- (7,0);
   \draw (6,0) -- (5,-1);
   \draw (5,1) node[anchor = east] {$(0,1)$};
   \draw (6,1) node[anchor = south west] {$(1,1)$};
   \draw (5.85,0) node[anchor = north west] {$(1,0)$}; 
   \filldraw (5,1) circle (1pt);
   \filldraw (6,1) circle (1pt);
   \filldraw (6,0) circle (1pt);   
  \end{tikzpicture}
 \end{center}
 Let us now consider the polynomial $g = t^4 + t^2x_1 -2t^2 x_2 + t^2x_1^2 - 2t^2x_2^2 + x_1x_2$.
 The Newton subdivisions of $f$ and $g$ are equal, but the tropicalization of $g$ has
 vertices $(2,0),(0,2)$ and $(2,2)$. Hence, we see that the Newton subdivision of a polynomial $f$ does
 not determine its tropicalization, but only its combinatorial type.

 If the Newton polytope is $2$-dimensional, we will therefore now mark each
 $2$-dimensional cell of the Newton subdivision with its corresponding vertex
 in $ \Trop(f) $ in order to fix the tropicalization.
\end{Example}

\begin{Definition}[Marked subdivision]
 Let $Q$ be a $2$-dimensional lattice polytope in $\mathbb R^2$ and $A = Q \cap
 \mathbb Z^2$. A \emph{marked subdivision} of $Q$ is a family $\{(A_i,y_i): i
 \in I \}$ such that
 \begin{enumerate}
  \item $ y_i \in \RR^2 $ for all $ i \in I $,
  \item each $A_i$ is a subset of $A$ such that $A_i$ contains exactly the vertices of $Q_i = \conv(A_i)$,
  \item $\{ Q_i: i \in I\}$ forms a subdivision of $Q$.
 \end{enumerate}
\end{Definition}

\begin{Definition}[Marked Newton subdivisions] \label{def:newtonsubdiv} ~
\begin{enumerate}
  \item \label{def:newtonsubdiv:a}
   Let $f = \sum_{|\nu| \leq d} a_\nu x^\nu \in \Ct[x_1,x_2]$ be a
   polynomial of degree $d$ such that $\Newt(f)$ is $2$-dimensional,
   and let $\{Q_i: i \in I\}$ be the Newton subdivision of $f$. For any $i \in
   I$, let $A_i \subset \mathbb Z^2$ be the vertices of $Q_i$ and $y_i$ the
   unique vector in $\mathbb R^2$ such that $\val(a_\mu) + \mu \cdot y_i =
   \val(a_\nu) + \nu \cdot y_i$ for all $\mu, \nu \in A_i$, which exists by
   Remark \ref{rem:condsub}. The point $y_i$ is then the vertex of $\Trop(f)$
   dual to $Q_i$. The marked subdivision $ \Subdiv(f) := \{(A_i,y_i) : i \in
   I\}$ will be called the \emph{marked Newton subdivision} of $f$.
 \item \label{def:newtonsubdiv:b}
   Let $C \subset \mathbb R^2$ be a tropical curve which is not a classical
   line, and let $ \{ Q_i : i \in I \} $ be its Newton subdivision. The
   \emph{marked Newton subdivision} $\Subdiv(C)$ of $C$ is defined as the
   subdivision $\{(A_i,y_i): i \in I\}$ of $\Newt(C)$ such that $ A_i $ is the
   set of vertices of $ Q_i $, and $ y_i $ the vertex of $C$ dual to $ Q_i $
   for all $ i \in I $.
\end{enumerate}
\end{Definition}

\begin{sidepic}{ex-2}
  \begin{Example}
 The marked Newton subdivision of the polynomial $f$ as in Example \ref{ex:newton} consists of the three
 pairs 
 \begin{itemize}
  \item $(\{(0,1),(1,1),(0,2)\},(1,0))$,
  \item $(\{(0,0),(1,0),(1,1),(0,1)\},(1,1))$, and 
  \item $(\{(1,0),(1,1),(2,0)\},(0,1))$.
 \end{itemize}
\end{Example}
\end{sidepic}

\vspace{1ex}

As expected, the following lemma shows that a tropical curve is completely
described by its marked Newton subdivision. We will use this fact in our
algorithm by comparing marked Newton subdivisions instead of tropical curves.

\begin{Lemma} \label{lemma:markednewton}
 Let $C$ be a tropical curve which is not a classical line, and let $f$ be a polynomial in 
 $\Ct[x_1,x_2]$ such that $f$ is not divisible by a monomial and $\Newt(f)$ is two-dimensional. 
 Then we have $\Trop(f) = C$ if and only if $\Subdiv(f)= \Subdiv(C)$. 
\end{Lemma}

\begin{proof}
 Assume that $\Subdiv(C)$ is the marked Newton subdivision of $f$.
 We have already seen in Proposition \ref{prop:dualsubdivision} that the Newton subdivision of 
 $f$ is dual to its tropicalization. Moreover, by definition, $C$ is dual to
 its marked Newton subdivision, hence $C$ and $\Trop(f)$ are of the same combinatorial type. 
 Additionally they have the same vertices, hence they coincide.

 Conversely, assume that $\Trop(f) = C$. As $f$ is not divisible by a monomial,
 we know by \cite[Lemma 4.14]{wir} that $\Newt(f) = \Newt(C)$. Moreover, since
 $\Trop(f)$ is dual to the Newton subdivision of $f$ we know by definition that
 the subdivisions coincide. In addition, the markings are exactly the vertices
 of $C$ and $\Trop(f)$, and thus are the same. So we have $\Subdiv(f) =
 \Subdiv(C)$.
\end{proof}

\begin{Remark} \label{rem-trop-deg}
 Using Lemma \ref{lemma:markednewton}, we see that tropicalization preserves degree: 
 We defined the degree of a tropical curve $C$ as the degree of its recession fan, \cf
 Definition \ref{def:degree}. Moreover, the Newton polytope of $C$ is defined as the Newton polytope
 of its recession fan, \cf Definition \ref{def:newtonsubdiv}. In \cite[Lemma 4.14]{wir}, it is shown
 that $\Newt(C)$ meets all three faces of $\conv((0,0),(d,0),(0,d))$, where $d
 = \deg(C)$,
 and since $\Newt(C) = \Newt(f)$, we know in particular that $\deg(f) = d$. 
\end{Remark}

\subsection{Computing Realizability}

In this section we will present an algorithm to decide whether or not a given
tropical curve in $ \Sigma = B(M(L))$ is relatively realizable in $L$, where as
above $X$ is a plane in $ \PP^n_{\Ct}$ and $L = I(X)$ can be generated by
polynomials in $K[x]$.

When working with the computer, we can only consider tropical curves with
vertices in $\mathbb Q^n$. The following proposition shows that for the
implementation of the relative realization problem we can in fact restrict to
the case of tropical curves with integer vertices.

\begin{Proposition} \label{prop-intvert}
  Fix $ m \in \QQ_{>0} $. Let $ C \subset \Sigma $ be a tropical curve which is
  not a classical line, and let $ C_m $ be the tropical curve obtained from $C$
  by rescaling $ y \mapsto my $ (preserving the weights). Moreover, let $ f \in
  \Ct[x] $ be a homogeneous polynomial without monomial factors, and let $ f_m
  $ be its image under the ring homomorphism $ \Ct[x] \to \Ct[x] $ sending $t$
  to $ t^m $.

  Then $ f_m $ realizes $ C_m $ if and only if $f$ realizes $C$.
\end{Proposition}

\begin{proof}
  By Theorem \ref{mainthm} it suffices to prove the statement for
  curves in $ \RR^3 / \RR \cdot \one \cong \RR^2 $, where we can use marked
  Newton subdivisions for the comparison. But if $ \{ (A_i,y_i) : i \in I \} $
  is the marked Newton subdivision of $f$ it follows from Definition
  \ref{def:newtonsubdiv} \ref{def:newtonsubdiv:a} that the marked Newton
  subdivision of $ f_m $ is just $ \{ (A_i, my_i) : i \in I \} $. As this
  corresponds to the curve obtained from $C$ by rescaling by a factor of $m$,
  the result follows.
\end{proof}

As already mentioned, it is the aim of this section to present an algorithm
able to decide whether or not a given tropical curve in $ \Sigma $ is
relatively realizable. To do so, we will use the idea to compare marked Newton
subdivisions. The following proposition shows a way to check if a given marked
subdivision is the marked Newton subdivision of a polynomial. As above, we set
$\val(0) = \infty$.

\begin{Lemma} \label{lem-subdiv}
  Let $S = \{(A_i,y_i): i \in I\}$ be the marked Newton subdivision of a
  tropical curve $C$ in $\mathbb R^2$ which is not a classical line,
  and let $f = \sum_\nu a_\nu x^\nu \in \Ct [x_1,x_2]$ be a polynomial with
  $\Newt(f) = \Newt(C)$.
  Then $ \Trop(f) = C $ if and only if the following conditions hold:
  \begin{enumerate}
  \item \label{lem-subdiv:a}
    For all $i \in I$ and all $\nu, \eta \in A_i$ we have
    $ \val(a_\nu) + \nu \cdot y_i = \val(a_\eta) + \eta \cdot y_i$, and
  \item \label{lem-subdiv:b}
    for any $\nu \in \Newt(f) \cap \mathbb N^2$ such that $\nu \notin A_i$ for
    all $i \in I$, choose $j \in I$ such that $\nu \in \conv(A_j)$ and $\eta
    \in A_j$ arbitrary, and require that $\val(a_\nu) + \nu\cdot y_j \geq
    \val(a_\eta) + \eta\cdot y_j$.
  \end{enumerate}
\end{Lemma}

\begin{proof}
  If $ \Trop(f)=C $ then we also have $ \Subdiv(f) = \Subdiv(C) $ by Lemma
  \ref{lemma:markednewton}, and hence the conditions \ref{lem-subdiv:a} and
  \ref{lem-subdiv:b} hold by Remark \ref{rem:condsub}.

  Conversely, assume that the conditions of the lemma hold. Let us consider the
  polyhedron
    $$ Q = \conv\left( (\nu, \lambda): \nu \in \bigcup_{i \in I} \, A_i,
       \lambda \geq \val(a_\nu) \right) \quad \subset \RR^2 \times \RR. $$
  We first show that projecting $Q$ to $\mathbb R^2$ induces the Newton
  subdivision of $C$: We know that $C$ is realizable since it is a plane
  tropical curve, hence there is a polynomial $ g = \sum_\nu b_\nu x^\nu \in
  \Ct[x_1,x_2]$ such that $\Subdiv(g) = S $. By Remark
  \ref{rem:condsub}, we know that for all $i \in I$ and $\nu, \eta \in A_i$, we
  have $\val(b_\nu) + y_i \cdot \nu = \val(b_\eta) + y_i \cdot \eta$. We
  assumed that $f$ fulfills the above conditions, so by \ref{lem-subdiv:a}
  there is a constant $c \in \mathbb R$ with $\val(a_\nu) = \val(b_\nu) + c$
  for all $\nu \in \bigcup_{i\in I} A_i$. Hence, the induced subdivision of $Q$
  equals the Newton subdivision of $g$ and thus the Newton subdivision of $C$.

  We now prove that $Q$ is the extended Newton polytope of $f$: Due to the form
  of $Q$ and using \ref{lem-subdiv:a}, we know that $(\nu,\lambda) \in Q$ if
  and only if $\nu \in \Newt(f)$ and $\lambda \geq \val(a_\mu) + y_i \cdot
  (\mu-\nu)$, where we have $i \in I$ with $\nu \in \conv(A_i)$ and $\mu \in
  A_i$ arbitrary. So, for $\nu \in \mathbb N^2 \cap \Newt(f)$, we know by
  \ref{lem-subdiv:b} that $(\nu,\val(a_\nu)) \in Q$. Thus $Q$ is the extended
  Newton polytope of $f$. In particular, the Newton subdivision of $f$ is the
  Newton subdivision of $C$.

  Following the definition of the marked Newton subdivision of a polynomial, we
  see with \ref{lem-subdiv:a} that the marking of $A_i$ in the Newton
  subdivision of $f$ is exactly $y_i$. So we have $\Subdiv(f) = S$, and it
  follows that $\Trop(f) = C$ by Lemma \ref{lemma:markednewton}.
\end{proof}

For a tropical curve $C$ in $ \RR^3 / \RR \cdot \one $ that is not a classical
line, Lemma \ref{lem-subdiv} can be used to compute all homogeneous polynomials
that tropicalize to $C$. However, if $C$ is a classical line the concept of
marked Newton subdivisions is not applicable. So in this case we should
describe the polynomials tropicalizing to $C$ separately.

\begin{Remark}[Realizability for classical lines] \label{rem:classline}
Let $C$ in $\mathbb R^2$ be a classical line and $y$ any point on $C$. 
Its Newton polytope is $1$-dimensional, \ie there are $\nu,\mu \in \mathbb N^2$ such that 
$\Newt(C) = \conv(\nu,\mu) \subset \mathbb R^2$. Let $m$ be the lattice length of $\Newt(C)$. 
Then the tropical curve $C$ is realizable by exactly the polynomials
$$f = \sum_{i=0}^m a_i x^{(1-\frac im)\nu + \frac im \mu} \in \Ct[x_1,x_2]$$
such that 
\begin{itemize}
 \item $\val(a_0) + \nu \cdot y = \val(a_m) + \mu \cdot y$ and
 \item $\val(a_i) + \left( (1-\frac im) \nu + \frac im \mu \right) \cdot y \geq \val(a_0) + \nu \cdot y$ 
for all $i = 1,\ldots,m-1$.
\end{itemize}
\end{Remark}

We are now ready to present our algorithm to check relative realizability. It
uses the matroid projections of Section \ref{sec:proj}. As already mentioned,
the concepts of this section can be transferred to the homogeneous case using
the identification of $ \RR^2 $ with $\mathbb R^3 / \mathbb R \cdot \mathbf 1 $
given by $ (x_1,x_2) \mapsto [0,x_1,x_2] $. The objects used in the algorithm
should be interpreted in this way.

In the following algorithm, we assume that we can perform calculations with
Puiseux series. We will explain afterwards in Remark \ref{remark:decomp} that
we can restrict ourselves in the computations to Puiseux series with fixed and
finitely many $t$-powers, so that the algorithm can actually be implemented on
a computer.

\begin{Notation}
  Let $f \in \Ct[x]$ be a polynomial, and let $B$ be a basis of $M(L)$. By
  $a_{B,\nu}$ we denote the coefficient of $x^\nu$ in the polynomial $f_B$ of
  Definition \ref{def:algproj}, \ie we have $f_B = \sum_{|\nu| = d} a_{B,\nu}
  x^\nu \in \Ct[x_i: i \in B]$.
\end{Notation}

\begin{alg}[Computing the relative realizability of curves in a tropical plane]
    \label{alg-1} ~

  INPUT: A homogeneous linear ideal $L$ in $\Ct[x]$ corresponding to a plane in
  $ \PP^n_{\Ct} $ with generators in $K[x]$, and a tropical curve $C$ in $
  \Sigma = B(M(L)) $

  OUTPUT: $1$ if the curve is relatively realizable in $L$, $-1$ otherwise

  ALGORITHM:
  \begin{enumerate}
  \item Compute the degree $d$ of $C$, \ie the degree of the recession fan of
    $C$, as \eg in \cite[Lemma 2.9]{wir}.
  \item Compute a basis $ (j_0,j_1,j_2) $ of $M(L)$.
  \item Set 
       $$ \qquad\quad
          f = \sum_{|\nu| = d} a_\nu x^\nu \in \Ct[x_{j_0},x_{j_1},x_{j_2}], $$
     where we consider the $ a_\nu \in \Ct $ as parameters (note that we need
     degree $d$ by Remark \ref{rem-trop-deg}, and that any homogeneous
     polynomial of degree $d$ can be written modulo $L$ in this form).
  \item \label{alg-1:d} For every basis $B$ of $M(L)$, do the following:
  \begin{enumerate}
   \item Compute the tropical push-forward $p_*^B (C)$ (see Remark
     \ref{rem-push}).
   \item Compute the polynomial $f_B \in \Ct [x_i: i \in B]$ as in Definition
     \ref{def:algproj}; its coefficients $ a_{B,\nu} $ are $K$-linear
     combinations of the parameters $ a_\nu $.
   \item For any $\nu \notin \Newt(p_*^B C)$, collect the condition $ a_{B,\nu}
     = 0 $ (as linear condition on the $ a_\nu $).
   \item If $p_*^B(C)$ is not a classical line, compute the marked Newton
     subdivision $S$ of $p_*^B (C)$ and collect the conditions on the
     valuations of the coefficients $a_{B,\nu}$ of $f_B$ equivalent to $S$
     being the marked Newton subdivision of $f_B$, \ie to $ p_*^B(C) =
     \Trop(f_B) $ (see Lemma \ref{lem-subdiv}):
 
     For any $(A,y) \in S$ with $A = \{\nu_1,\ldots,\nu_t\}$, we get
     equalities: 
       $$ \qquad\qquad\qquad
          \val(a_{B,\nu_i}) + y \cdot \nu_i
          = \val(a_{B,\nu_1}) + y \cdot \nu_1 \quad
          \text{for all $ i \in \{2,\ldots,t\} $.} $$
     For any lattice point $\eta$ in $\conv(A)$ which is not contained in $A$,
     we get the condition:
       $$ \qquad\qquad\qquad
          \val(a_{B,\eta}) + y \cdot \eta
          \geq \val(a_{B,\nu_1}) + y \cdot \nu_1. $$ 
   \item If $p_*^B(C)$ is a classical line, collect the conditions on $f_B$
     from Remark \ref{rem:classline}.
  \end{enumerate}

  \item Check if the collected conditions have a common solution. If such a
    solution exists, $C$ is realizable by Theorem \ref{mainthm}, so return $1$.
    If there is no common solution, the tropical curve $C$ is not relatively
    realizable in $L$, hence set the output to $-1$. 
 \end{enumerate}
\end{alg}

\begin{Remark}[Computation of the push-forwards of $C$] \label{rem-push}
  To compute the tropical push-forwards $p_*^B (C)$ in Algorithm \ref{alg-1} in
  the easy case when the tropical curve $C$ is contained in the matroidal fan
  $ \Sigma $, we follow the canonical attempt: We naively project any
  polyhedron of the tropical variety, possibly refine them so they fit together
  with the other images, forget about those mapping to lower-dimensional
  polyhedra, and combine polyhedra mapping to the same maximal-dimensional
  polyhedron in the projection by summing up their multiplicities. This is
  illustrated in the following example.
\end{Remark}

\begin{Example}
 Consider the tropical curve $C$ in $ \Sigma = \Trop(x_0+x_1+x_2+x_3) $ shown
 in the picture below, with vertices $ [0,1,1,0],[0,0,0,2],[1,0,0,1]$, and
 $[1,0,0,3]$. Its tropical push-forward $p_*^{(0,1,2)} C$ is shown on the
 right. If we do not indicate a weight of an edge, the weight is $1$.

 \def\message#push{}%
  {\begin{center}\input{push}\end{center}}

 For the computation of the tropical push-forward $p_*^{(0,1,2)} C$ we have to
 refine the polyhedron $\conv([1,0,0,1],[0,1,1,0])$ to the two polyhedra
 $\conv([1,0,0,1],[0,0,0,0])$ and $\conv([0,0,0,0],[0,1,1,0])$. With this
 refinement, the cells $\conv([1,0,0,1],[0,0,0,0])$ and
 $\conv([1,0,0,3],[0,0,0,2])$ of $C$ both map to the cell
 $\conv([1,0,0],[0,0,0])$ in $p_*^{(0,1,2)} C$. So, this cell has multiplicity
 $2$. Moreover, the cell $ [1,0,0] + \cone([1,0,0]) $ also has multiplicity $2$
 since there are two cells with multiplicity $1$ mapping to it. Since we only
 consider tropical curves up to refinement, we did not indicate this refinement
 of the cell $\cone([1,0,0])$ of $ p_*^{(0,1,2)} C $ in our picture. The
 polyhedra $ \conv([1,0,0,1],[1,0,0,3])$ and $[1,0,0,3] + \cone([0,0,0,1])$
 both map to the point $[1,0,0]$, so to a polyhedron of lower dimension. Hence,
 these cells do not contribute to the push-forward.
\end{Example}

\begin{Remark}[Fixing an initial valuation] \label{rem:initialval}
  For any $c \in \mathbb R$, we have $L+(f) = L + (t^{-c}f)$. Given a tropical
  curve $C$ in $B(M(L))$, fixing an initial basis $ (j_0,j_1,j_2) $ of $M(L)$
  as in Algorithm \ref{alg-1} and a vertex $\eta$ of $\Newt(p_*^{(j_0,j_1,j_2)}
  C)$, we might thus assume that for any polynomial $f \in
  \Ct[x_{j_0},x_{j_1},x_{j_2}]$ realizing $C$, the coefficient of $x^\eta$ in
  $f$ has valuation zero. This assumption transforms the conditions in
  Algorithm \ref{alg-1} \ref{alg-1:d} into conditions of the form
  \begin{align*}
    a_{B,\nu} &= 0 \\
    \text{or} \quad
    \val(a_{B,\nu}) &= c_\nu^B \tag{$*$} \\
    \text{or} \quad
    \val(a_{B,\nu}) &\geq c_\nu^B
  \end{align*}
  for some $c_\nu^B \in \mathbb R$ and all bases $B$ of $M(L)$, depending on
  whether $\nu$ is contained in the Newton polytope of the projection $ p_*^B C
  $, and whether it is a vertex of a polytope in the corresponding Newton
  subdivision.
\end{Remark}

\begin{Notation} \label{notation:rhs} ~
  \begin{enumerate}
  \item For a polynomial $ f = \sum_{|\nu| \leq d} a_\nu x^\nu \in \Ct[x] $,
    let $a_\nu^k$ be the coefficient of $t^k$ in $a_\nu$, \ie we have $a_\nu =
    \sum_{k \in \mathbb R} a_\nu^k t^k$.
  \item The set of all right hand sides $c_\nu^B$ in Remark
    \ref{rem:initialval} will in the following be denoted by $\RHS(C)$. It is
    obviously a finite subset of $ \RR $.
  \end{enumerate}
\end{Notation}

\begin{Remark}[Reducing Puiseux series to finitely many coefficients]
    \label{remark:decomp}
  In Algorithm \ref{alg-1}, we start by setting $f = \sum_{|\nu| = d} a_\nu
  x^\nu$ and consider its coefficients as parameters. We then get conditions on
  the valuations of $a_{B,\nu}$, where $B$ is a basis of $M(L)$. However, it is
  not clear how to work with elements in $\Ct$ and thus how we can check that
  these conditions have a common solution. We will now see that we can simplify
  the parameters $a_\nu \in \Ct$ such that they only contain fixed and finitely
  many powers of $t$, namely $ t^k $ for $ k \in \RHS(C) $.

  For any basis $B$ of $ M(L) $, $ k \in \RR $, and $ \mu \in \NN^2 $, note
  that the $ t^k \, x^\mu $-coefficient $ a_{B,\mu}^k $ of $ f_B $ is a
  $K$-linear combination of the $ t^k $-coefficients $ \{ a_\nu^k : |\nu| = d
  \} $ since $L$ is generated by polynomials over $K$. Hence if $f$ satisfies
  the conditions $(*)$ of Remark \ref{rem:initialval}, the polynomial
    \[ g = \sum_{|\nu| = d} b_\nu x^\nu
       \quad\text{with}\quad
       b_\nu = \sum_{k \in \RHS(C)} a_\nu^k t^k \]
  has the same $ t^k $-terms as $f$ for all $ k \in \RHS(C) $, and thus
  satisfies the conditions $(*)$ as well. To check relative realizability we
  can therefore assume from the beginning that $ f = \sum_{|\nu|=d}
  \sum_{k \in \RHS(C)} a_\nu^k t^k \, x^\nu $, and consider the finitely many
  coefficients $ a_\nu^k \in K $ as our new set of parameters. The conditions
  $(*)$ can then be translated into equations $ a_{B,\nu}^k = 0 $ (for the
  condition $ a_{B,\nu} = 0 $ or if $ k<c_\nu^B $) and inequalities $
  a_{B,\nu}^k \neq 0 $ (for the condition $ \val(a_{B,\nu}) = c_\nu^B $ if $
  k=c_\nu^B $), where as explained above the $ a_{B,\nu}^k $ are fixed
  $K$-linear combinations of the $ a_\nu^k $. This finally allows to implement
  Algorithm \ref{alg-1} on a computer. 
\end{Remark}

In particular, we see that the conditions on the parameters $\{a_\nu^k: |\nu| =
d, k \in \RHS(C)\}$ can be decomposed into conditions on the parameters
$\{a_\nu^k: |\nu| = d\}$ for all $k \in \RHS(C)$. These conditions on the
$a_\nu^k$ for fixed $k$ are parts of the conditions for tropical fan curves.
So $C$ is relatively realizable if and only if there are tropical fan curves
for each $ k \in \RHS(C) $ satisfying the corresponding conditions. In other
words, our relative realizability question for non-constant coefficient
tropical curves can be decomposed into several relative realizability questions
for tropical fan curves, \ie constant coefficient tropical curves.

\begin{Remark}[Implementation in Singular] \label{rem-singular}
  The algorithm described above is implemented in the Singular library
  ``realizationMatroidsNC.lib'' \cite{Win12}. Let us explain how to use this
  library. As in the software package \emph{gfan} by Anders Jensen \cite{gfan},
  a tropical curve in $ \RR^{n+1}/\RR \cdot \one $ is given by its vertices,
  the directions of its recession fan, its edges, and the multiplicities on the
  edges. More precisely, a tropical curve $C$ is completely described by the
  following three lists:
  \begin{enumerate}
  \item $V = ((1,v_1),\ldots,(1,v_s),(0,r_1),\ldots,(0,r_t))$,  where
    $v_1,\ldots,v_s$ are the vertices of $C$ and $r_1,\ldots,r_t$ the
    primitive vectors of the cones in $\Rec(C)$, both in homogeneous
    coordinates, and both normalized with minimum $0$ over the coordinates.
  \item $E = ( (i,j): $ there is a bounded edge between $ v_i $ and $ v_j $ or
    there is an unbounded cell starting at $ v_i $ in the direction of $
    r_{j-s} $).
  \item $M = (m_1,\ldots,m_k)$, where $m_i$ is the weight of the maximal cell
    of $C$ at the $i$-th entry of $E$.
  \end{enumerate}
  For our algorithm, it is not necessary to order the elements in $V$ as above.
  It is only important that the entries in the tuples in $E$ fit to the
  positions of the corresponding elements in $V$. Moreover, the primitive
  directions may appear only once in $V$, no matter how many unbounded cells in
  $C$ leave in this direction.

  \vspace{1ex}

  \begin{sidepic}{ex-3}
    Let us consider the example curve $ C \subset \Trop(L) $ shown on the
    right: we have $ L = (x_0+x_1+x_2+x_3) $, and the curve $C$ has weight $1$
    on all maximal cells, the unit vectors as unbounded directions, and
    vertices $[0,1,1,0]$ and $[1,0,0,1]$. The following Singular code then
    shows how to check the relative realizability of $C$ as in Algorithm
    \ref{alg-1}, and how to compute a polynomial realizing $C$ if this is the
    case. Note that the entries of $V$ are restricted to integers because of
    Proposition \ref{prop-intvert}.
  \end{sidepic}

  
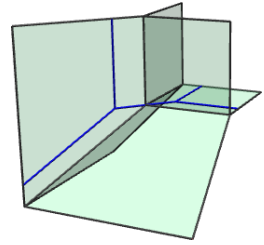
\vspace{1ex}

  \begin{verbatim}
  > LIB "realizationMatroidsNC.lib";
  > ring R = (0,t),(x0,x1,x2,x3),dp;
  > ideal L = x0+x1+x2+x3;
  > list V = list(intvec(1,0,1,1,0), intvec(1,1,0,0,1),
    intvec(0,0,1,0,0), intvec(0,0,0,1,0),
    intvec(0,0,0,0,1), intvec(0,1,0,0,0));
  > list E = list(intvec(1,2),intvec(1,3),intvec(1,4),
    intvec(2,5),intvec(2,6));
  > list M = list(1,1,1,1,1);
  > list C = list(V,E,M);
  > realizable(L,C);
  1
  > realizablePoly(L,C);
  1 (t)*x0+x1+(t+1)*x2
  \end{verbatim}
\end{Remark}

\subsection{The space of relatively realizable curves}

Given a tropical curve in $\Sigma$, we can check if it is relatively realizable
with our algorithm above. However, one may also want to know if the space of
all relatively realizable tropical curves in $\Sigma$ carries a nice structure.
To investigate this question, we first have to study the space of all tropical
curves in $\Sigma$. 

Firstly, consider the space of tropical curves in $ \mathbb R^n / \RR \cdot
\one $. Fixing a combinatorial type, including the directions of all edges, any
curve of this type can be described by the position of one of its vertices and
the lattice lengths of its bounded edges (which might have to satisfy some
linear conditions in order for loops to be closed). Requiring the curves to lie
in $ \Sigma $ imposes some additional linear conditions. Hence the space of all
tropical curves in $ \Sigma $ of a given combinatorial type is an open
polyhedron. Its boundary, obtained by shrinking some of the edge lengths
to $0$, corresponds to curves of different combinatorial types.

For a given degree $ d \in \NN_{>0} $, let us denote by $ T_d(L) $ the space of
all tropical curves of degree $d$ in $L$, obtained by gluing the above
polyhedra together according to the degeneration of the corresponding
combinatorial types. Hence $ T_d(L) $ can be written as a finite union of
polyhedra, glued together at their faces along linear maps. We will refer to
such a structure as an \emph{abstract polyhedral set}. Note that it does not
come with a natural embedding in a real vector space.

Inside $ T_d(L) $, denote by $ R_d(L) $ the subset of all relatively realizable
tropical curves in $ \Sigma $. We will prove that $ R_d(L) $ is closed in $
T_d(L) $, and that it is in fact an abstract polyhedral set itself.

\begin{Lemma} \label{lemma:closed}
  $ R_d(L) $ is closed in $ T_d(L) $ for all $d$.
\end{Lemma}

\begin{proof}
  We prove that the space of tropical curves in $\Sigma$ which are not
  relatively realizable is open in $T_d(L)$. Let $C$ be a tropical curve in
  $\Sigma$ which is not relatively realizable and let $c_{\nu}^B$ by the right
  hand sides of the conditions $\val(a_{B,\nu}) = c_{\nu}^B$ or
  $\val(a_{B,\nu}) \geq c_{\nu}^B$, respectively, as in Remark
  \ref{rem:initialval}. For simplicity, we set $ c_\nu^B = \infty $ if the
  corresponding condition is $ a_{B,\nu} = 0 $.
  
  Since $C$ is not relatively realizable, there is a basis $B$ of $M(L)$ and a
  vertex $\mu$ of a polytope in the Newton subdivision of $p_*^B (C)$ such that
  the required condition $ a_{B,\mu}^k \neq 0 $ for $ k = c_\mu^B $ is in
  contradiction to the required equations $ a_{A,\nu}^k = 0 $ for all bases
  $A$ and lattice points $ \nu $ with $ c_\nu^A > k = c_\mu^B $.

  For a small deformation of $C$, two things can happen:
  \begin{itemize}
  \item A vertex of valence greater than $3$ in one (or more) of the
    projections $ p_*^A C $ may be resolved. Correspondingly, the Newton
    subdivision of $ p_*^A C $ is subdivided further. This means that some
    conditions $\val(a_{A,\nu}) \geq c_{\nu}^A$ change into
    $ \val(a_{A,\nu}) = c_{\nu}^A $. If there was a contradiction with the
    original conditions, there will still be a contradiction with these
    stronger conditions.
  \item The right hand sides of the conditions in Remark \ref{rem:initialval}
    slightly change. However, if we had $ c_{\nu}^A > c_{\mu}^B $ originally,
    this inequality still holds after a small deformation of these numbers.
    Hence, the contradiction is still present and thus the deformed curve is
    not relatively realizable either.
  \end{itemize}
  Altogether, this means that the space of tropical curves in $ \Sigma $ that
  are not relatively realizable is open in $ T_d(L) $.
\end{proof}

\begin{Proposition} \label{prop-polyhedral}
  $ R_d(L) $ is an abstract polyhedral set for all $d$.
\end{Proposition}

\begin{proof}
  Continuing the notation of the proof of Lemma \ref{lemma:closed}, let $C$ be
  a relatively realizable curve, and let $ c_\nu^B $ be the right hand sides of
  the conditions in Remark \ref{rem:initialval}. Consider the subset $S$ of $
  T_d(L) $ of all curves $ \tilde C $ such that
  \begin{enumerate}
  \item $ p_*^B \tilde C $ has the same combinatorial type as $ p_*^B C $ for
    all bases $B$ of $ M(L) $, and
  \item the right hand sides $ \tilde c_\nu^B $ satisfy the same equalities and
    inequalities as the original ones, \ie for all bases $A,B$ and lattice
    points $ \mu,\nu $ we have $ \tilde c_\mu^A = \tilde c_\nu^B $ if and only
    if $ c_\mu^A = c_\nu^B $, and $ \tilde c_\mu^A < \tilde c_\nu^B $ if and
    only if $ c_\mu^A < c_\nu^B $.
  \end{enumerate}
  As these right hand sides depend linearly on the vertices, and thus also on
  the edge lengths of the curves, this subset $ S \subset T_d(L) $ is an open
  polyhedron. Moreover, by the form of the conditions for relative
  realizability all curves in $S$ are relative realizable. Hence $ R_d(L) $ is
  a finite union of open polyhedra, and so the proposition follows by Lemma
  \ref{lemma:closed}.
\end{proof}

 % Relative Realizability
  \section{General criteria for relative realizability} \label{sec-criteria}

As before, let $K$ be any algebraically closed field. Moreover, let $L \subset
\Ct[x]$ be the ideal of a plane $ X \subset \PP^n_{\Ct} $ not contained in any
coordinate hyperplane and with generators in $K[x]$, and let $M(L)$ be the
corresponding matroid. As a warm-up, we want to reprove the known fact that if
a tropical curve $C$ in $ \Sigma = B(M(L)) $ is relatively realizable in $L$,
then so is its recession fan, see for instance \cite[Theorem 3.5.6]{MS}.

\begin{Proposition} \label{prop:critrecession}
  Let $C$ be a tropical curve in $ \Sigma $ which is relatively realizable in
  $L$. Then its recession fan $\Rec(C)$ is also relatively realizable in $L$.
\end{Proposition}

\begin{proof}
  Assume that $\Rec(C)$ is not relatively realizable in $L$. Since $\Rec(C)$ is
  a fan, so is any push-forward $p_*^B \Rec(C)$ for a basis $B$ of $ M(L) $. In
  particular, the Newton subdivision of $p_*^B \Rec(C)$ consists of one
  polytope $\Newt(p_*^B \Rec(C))$. If $p_*^B \Rec(C)$ is not a classical line,
  then $\Newt(p_*^B \Rec(C))$ is marked with the origin in the sense of
  Definition \ref{def:newtonsubdiv}. Hence, the conditions $(*)$ of Remark
  \ref{rem:initialval} for $ \Rec(C) $ are of the form $ a_{B,\nu} = 0 $, $
  \val(a_{B,\nu}) = 0 $, or $ \val(a_{B,\nu}) \ge 0 $. The same is true by
  Remark \ref{rem:classline} if $ p_*^B \Rec(C) $ is a classical line.

  As $ \Rec (C) $ is not realizable, there must be a contradiction in these
  conditions. In other words, there must be an $ a_{A,\nu} $ with condition $
  \val(a_{A,\nu}) = 0 $ which is a linear combination of other $ a_{B,\mu} $
  with conditions $ a_{B,\mu} = 0 $. But then $ \nu $ must be a vertex of $
  \Newt(p_*^A \Rec(C)) = \Newt(p_*^A C) $, and the $ \mu $ are outside $
  \Newt(p_*^B \Rec(C)) = \Newt(p_*^B C) $. Hence for $C$ we again have the
  conditions $ \val(a_{A,\nu}) = c_\nu^A $ and $ a_{B,\mu} = 0 $ for the same $
  A,B,\nu,\mu $ as above, which is still a contradiction. It follows that $C$
  is not realizable either.
\end{proof}

\begin{sidepic}{plane}
  In general, it seems to be difficult to work out non-algorithmic rules to
  decide whether a given curve in $ \Sigma $ is relatively realizable. However,
  we will now prove some results in the first interesting case when $X$ is the
  plane in $ \PP^3_{\Ct} $ with ideal $ L=(x_0+x_1+x_2+x_3) $. In this case, we
  will denote its tropicalization $ \Sigma $ by $ L^3_2 $, as shown in the
  picture on the right. Moreover, let us fix the following notation for the
  projections that we will need.
\end{sidepic}

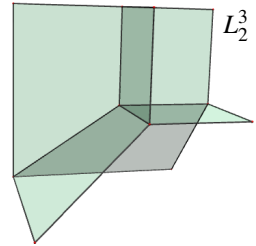
\begin{Notation}[Projections and Newton polytopes] \label{not:13} ~
\begin{enumerate}
 \item Let $f$ be a homogeneous polynomial of degree $d$ in
   $ \Ct[x] := \Ct[x_0,x_1,x_2,x_3]$. 
 \begin{itemize}
  \item We define $f_3 := f_{(0,1,2)} \in \Ct[x_0,x_1,x_2]$ and
    $f_1 := f_{(0,2,3)} \in \Ct[x_0,x_2,x_3]$, with the notation as in
    Definition \ref{def:algproj}.
  \item For $ i \in \{1,3 \}$ let $\Delta_i$ be the Newton polytope of $f_i$.
  \item We denote the coefficients of $f_3$ and $ f_1 $ by $ a_\nu $ and $
    b_\nu $, respectively. As we always use non-homogeneous coordinates for
    Newton polytopes, this means that
      \[ \qquad\qquad\qquad
         f_3 = \sum_{i+j \le d} a_{(i,j)} \, x_0^{d-i-j} x_1^i x_2^j
         \quad\text{and}\quad
       f_1 = \sum_{i+j \le d} b_{(i,j)} \, x_0^{d-i-j} x_2^i x_3^j. \]
 \end{itemize}
 \item Let $C$ be a tropical curve in $L_2^3$.
 \begin{itemize}
  \item We define $C_3 = p_*^{(0,1,2)} C$ and $C_1 = p_*^{(0,2,3)} C$.
  \item The Newton polytope of $C_i$ will in the following be denoted by $P_i$,
    $i=1,3$.
 \end{itemize}
\end{enumerate}
\end{Notation}

In order to prove obstructions to relative realizability, we aim to find
relations between the Newton polytopes $\Delta_1$ and $\Delta_3$, \ie relations
between the polynomials $f_1$ and $f_3$. If the Newton polytopes $P_1$ and
$P_3$ of a tropical curve $C$ do not satisfy these relations, we can use this
to prove that $C$ cannot be relatively realizable. 

\begin{Notation} \label{notation:binomial}
 For $n \in \mathbb Z$ and $k \in \mathbb N$, let $\binom{n}{k} = \frac{1}{k!}
 \, n\cdot(n-1)\cdots(n-k+1)$. For $k \in \mathbb Z_{<0}$, we set $\binom nk =
 0$. Note that $\binom nk = 0$ if $n > 0$ and $k > n$.
\end{Notation}

Using induction on $a$, it is an easy calculation to prove the following lemma.

\begin{Lemma} \label{lemma:equationbinomial}
 For $a,b,c \in \mathbb N$, we have 
   $$ \sum_{j=0}^a (-1)^j \binom aj \binom{b+j}c = (-1)^a \binom{b}{c-a}
      \quad\text{and}\quad
      \sum_{j=0}^a (-1)^j \binom aj \binom{b-j}c = \binom{b-a}{c-a}.$$
\end{Lemma}

\begin{Lemma} \label{lemma:equpoly}
  Let $ f \in \Ct[x] $ be a homogeneous polynomial of degree $d$. Choose $ k,n
  \in \mathbb N$ with $ k+n \leq d$, $l \in \mathbb N$ with $l \leq n+1$,
  and $m \in \mathbb N$ with $m \leq k+1$.

  As in the following picture, we denote by $\Lambda_1$ the set of all $(i,j)$
  such that $k \leq i \leq d-n$ and either $j < l$ or $j \geq d-i-n+l$. By
  $\Lambda_3$, we denote the set of all $(d-t-s,s) \in \NN^2 $ for $ t,s \in
  \NN $ with $n \leq t \leq d-k$ and either $s < m$ or $s \geq d-t-k+m$.

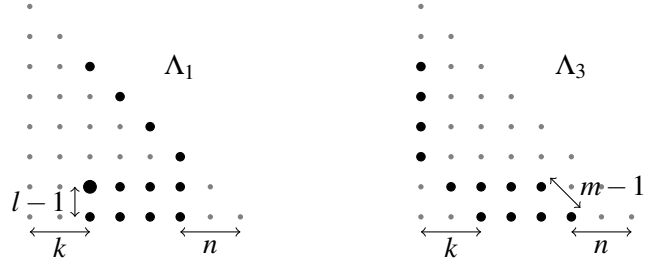
  \begin{center} \begin{tikzpicture}[scale=0.4]
    \filldraw[color=gray] (0,0) circle (2pt);
    \filldraw[color=gray] (1,0) circle (2pt);
    \filldraw (2,0) circle (4pt);
    \filldraw (3,0) circle (4pt);
    \filldraw (4,0) circle (4pt);
    \filldraw (5,0) circle (4pt);
    \filldraw[color=gray] (6,0) circle (2pt);
    \filldraw[color=gray] (7,0) circle (2pt);
    \filldraw[color=gray] (0,1) circle (2pt);
    \filldraw[color=gray] (1,1) circle (2pt);
    \filldraw (2,1) circle (6pt);
    \filldraw (3,1) circle (4pt);
    \filldraw (4,1) circle (4pt);
    \filldraw (5,1) circle (4pt);
    \filldraw[color=gray] (6,1) circle (2pt);
    \filldraw[color=gray] (0,2) circle (2pt);
    \filldraw[color=gray] (1,2) circle (2pt);
    \filldraw[color=gray] (2,2) circle (2pt);
    \filldraw[color=gray] (3,2) circle (2pt);
    \filldraw[color=gray] (4,2) circle (2pt);
    \filldraw (5,2) circle (4pt);
    \filldraw[color=gray] (0,3) circle (2pt);
    \filldraw[color=gray] (1,3) circle (2pt);
    \filldraw[color=gray] (2,3) circle (2pt);
    \filldraw[color=gray] (3,3) circle (2pt);
    \filldraw (4,3) circle (4pt);
    \filldraw[color=gray] (0,4) circle (2pt);
    \filldraw[color=gray] (1,4) circle (2pt);
    \filldraw[color=gray] (2,4) circle (2pt);
    \filldraw (3,4) circle (4pt);
    \filldraw[color=gray] (0,5) circle (2pt);
    \filldraw[color=gray] (1,5) circle (2pt);
    \filldraw (2,5) circle (4pt);
    \filldraw[color=gray] (0,6) circle (2pt);
    \filldraw[color=gray] (1,6) circle (2pt);
    \filldraw[color=gray] (0,7) circle (2pt);

    \draw (5,5) node {$\Lambda_1$};
    \draw[<->] (0,-0.5) -- (2,-0.5);
    \draw (1,-1) node {$k$};
    \draw[<->] (5,-0.5) -- (7,-0.5);
    \draw (6,-1) node {$n$};
    \draw[<->] (1.5,0) -- (1.5,1);
    \draw (0.3,0.5) node {$l-1$};

    \begin{scope}[xshift = 3cm]
    \filldraw[color=gray] (10,0) circle (2pt);
    \filldraw[color=gray] (11,0) circle (2pt);
    \filldraw (12,0) circle (4pt);
    \filldraw (13,0) circle (4pt);
    \filldraw (14,0) circle (4pt);
    \filldraw (15,0) circle (4pt);
    \filldraw[color=gray] (16,0) circle (2pt);
    \filldraw[color=gray] (17,0) circle (2pt);
    \filldraw[color=gray] (10,1) circle (2pt);
    \filldraw (11,1) circle (4pt);
    \filldraw (12,1) circle (4pt);
    \filldraw (13,1) circle (4pt);
    \filldraw (14,1) circle (4pt);
    \filldraw[color=gray] (15,1) circle (2pt);
    \filldraw[color=gray] (16,1) circle (2pt);
    \filldraw (10,2) circle (4pt);
    \filldraw[color=gray] (11,2) circle (2pt);
    \filldraw[color=gray] (12,2) circle (2pt);
    \filldraw[color=gray] (13,2) circle (2pt);
    \filldraw[color=gray] (14,2) circle (2pt);
    \filldraw[color=gray] (15,2) circle (2pt);
    \filldraw (10,3) circle (4pt);
    \filldraw[color=gray] (11,3) circle (2pt);
    \filldraw[color=gray] (12,3) circle (2pt);
    \filldraw[color=gray] (13,3) circle (2pt);
    \filldraw[color=gray] (14,3) circle (2pt);
    \filldraw (10,4) circle (4pt);
    \filldraw[color=gray] (11,4) circle (2pt);
    \filldraw[color=gray] (12,4) circle (2pt);
    \filldraw[color=gray] (13,4) circle (2pt);
    \filldraw (10,5) circle (4pt);
    \filldraw[color=gray] (11,5) circle (2pt);
    \filldraw[color=gray] (12,5) circle (2pt);
    \filldraw[color=gray] (10,6) circle (2pt);
    \filldraw[color=gray] (11,6) circle (2pt);
    \filldraw[color=gray] (10,7) circle (2pt);

    \draw (15,5) node {$\Lambda_3$};
    \draw[<->] (15,-0.5) -- (17,-0.5);
    \draw (16,-1) node {$n$};
    \draw[<->] (15.3,0.3) -- (14.3,1.3);
    \draw (16.4,1) node {$m-1$};
    \draw[<->] (10,-0.5) -- (12,-0.5);
    \draw (11,-1) node {$k$};
    \end{scope}
  \end{tikzpicture} \end{center}

  Then the coefficients $ a_\nu $ and $ b_\nu $ of $f_3$ and $f_1$ as in
  Notation \ref{not:13} satisfy the relation
    $$ \sum_{\nu \in \Lambda_1} \beta_\nu b_\nu
       = \sum_{\nu \in \Lambda_3} \alpha_\nu a_\nu $$
  for some $ \alpha_\nu, \beta_\nu \in \ZZ $ with $\beta_{(k,l-1)} = 1$. 
\end{Lemma}

\begin{proof}
  As $ f_1(x_0,x_2,x_3) = f_3(x_0,-x_0-x_2-x_3,x_2) $, it is an easy
  calculation to see that
    $$ b_{(i,j)} = \sum_{s=0}^{i} \sum_{t=0}^{d-i-j} (-1)^{d-t-s}
       \binom{d-t-s}{i-s} \binom{d-i-t}{j} a_{(d-t-s,s)}. $$
  We consider the linear combination
   $$ S = \sum_{i=k}^{d-n} \;\sum_{j=0}^{d-i} (-1)^{i-k+l-1-j}
      \binom{i-m}{i-k} \binom{d-i-n+l-1-j}{d-i-n} b_{(i,j)}. $$
  Note that $\binom{d-i-n+l-1-j}{d-i-n} = 0$ if $ l \leq j < d-i-n+l$. Hence
  $S$ is a linear combination of $ b_{(i,j)} $ with $ (i,j) \in \Lambda_1 $.
  Moreover, the coefficient of $b_{(k,l-1)}$ in $S$ is $1$.

  Replacing $b_{(i,j)}$ in $S$ by the above equality, we get
  \begin{align*}
    S = \sum_{i=k}^{d-n} \; \sum_{j=0}^{d-i} \; \sum_{s=0}^i \sum_{t=0}^{d-i-j}
    &(-1)^{i-k+l-1-j+d-t-s} \binom{i-m}{i-k} \\
    &\cdot \binom{d-i-n+l-1-j}{d-i-n} \binom{d-t-s}{i-s} \binom{d-i-t}{j}
    a_{(d-t-s,s)}.
  \end{align*}
  Changing the summation order, we get
  $ S = \sum_{s=0}^{d-n} ~ \sum_{t=0}^{\min(d-k,d-s)} ~
  \sum_{i=\max(k,s)}^{\min(d-n,d-t)} ~ \sum_{j=0}^{d-i-t} \big( \ldots \big) $.
  This expression simplifies by Lemma \ref{lemma:equationbinomial} to
  \begin{align*}
    S = \sum_{s=0}^{d-n} ~ \sum_{t=0}^{\min(d-k,d-s)} ~
      \sum_{i=\max(k,s)}^{\min(d-n,d-t)}
    &(-1)^{i-k+l-1+d-t-s} \binom{i-m}{i-k} \\
    &\cdot \binom{d-t-s}{i-s} \binom{t+l-n-1}{t-n} a_{(d-t-s,s)}.
  \end{align*}
  We have to check that this expression contains only $ a_\nu $ with $ \nu \in
  \Lambda_3 $. For $t < n$ we have $\binom{t+l-n-1} {t-n} =0$, and thus
  $a_{(d-t-s,s)}$ does not appear in $S$. It is obvious that $ a_{(d-t-s,s)} $
  also does not appear in $S$ if $ t>d-k $. For $i < k$ we have
  $\binom{i-m}{i-k} = 0$, so we can consider the last sum as starting at $i =
  s$. For $n \leq t \leq d-k$, the coefficient of $a_{(d-t-s,s)}$ in $S$ is
  then given by
  \begin{align*}
    \sum_{i=s}^{d-t} &(-1)^{i-k+l-1+d-t-s} \binom{i-m}{i-k}
      \binom{d-t-s}{i-s} \binom{t+l-n-1}{t-n} \\
    = & \binom{t+l-n-1}{t-n} \sum_{i=0}^{d-t-s}
        (-1)^{i-k+l-1+d-t} \binom{i+s-m}{i+s-k} \binom{d-t-s}i \\
    = & \binom{t+l-n-1}{t-n} \sum_{i=0}^{d-t-s}
        (-1)^{i-k+l-1+d-t} \binom{s-m+i}{k-m} \binom{d-t-s}i \\
    = & \binom{t+l-n-1}{t-n} (-1)^{i-k+l-1+s} \binom{s-m}{t-d+k+s-m}.
  \end{align*}
  This coefficient is zero for all $(d-t-s,s)$ with $ n \leq t \leq d-k$ and $
  m \leq s < d-t-k+m$. Hence the linear combination $S$ only involves
  $a_\nu$ for $ \nu \in \Lambda_3$.
\end{proof}

This dependency between $f_1$ and $f_3$ will be used to prove our obstructions
to relative realiza\-bility: in the proof of Proposition \ref{prop:andreas} we
show that for certain $k,n,l,$ and $m$, all $a_\nu$ and $b_\nu$ appearing in
the above equation but $b_{(k,l-1)}$ are zero since the corresponding exponents
are not contained in the Newton polytopes $\Delta_1$ and $\Delta_3$,
respectively. Using Lemma \ref{lemma:equpoly}, we then see that we must have
$b_{(k,l-1)} = 0$. 

In \cite{wir}, we also used relations between $f_1$ and $f_3$ to prove
obstructions to relative realizability of tropical fan curves. Lemma
\ref{lemma:equpoly} generalizes most of these relations: it implies \cite[Lemma
5.9]{wir}, \cite[Lemma 5.15]{wir}, and \cite[Proposition 5.17]{wir}. 

\begin{Notation} \label{not-numbers}
  Let $C$ be a tropical curve of degree $d$ in $L_2^3$. In the following, we
  will refer to the diagonals in $ P_3 $ and the vertical lines in $ P_1 $ as
  \emph{rows}, counting from $0$ with the row of lattice length $d$. As
  illustrated in the picture in the introduction, for a vertex $
  \mu=(\mu_1,\mu_2) $ of $ P_3 $ we set
  \begin{align*}
    s_\mu &= d-\mu_1-\mu_2 \;\; \text{(the row number of $ \mu $)}, \\
    n_\mu &= |\{ \nu \in \NN^2 : \nu_1 + \nu_2 = \mu_1+\mu_2, \nu \notin P_3
      \}| \;\; \text{(lattice points not in $ P_3 $ in the row of $ \mu $)}, \\
    r_\mu &= d+1-s_\mu-n_\mu \;\; \text{(lattice points in $ P_3 $ in the row
      of $ \mu $)}, \\
    l_\mu &= |\{ \nu \in \NN^2: \nu_1 = n_\mu, \nu \in P_1\}| \;\;
      \text{(lattice points in $ P_1 $ in row number $ n_\mu $)}.
  \end{align*}
  Correspondingly, for a vertex $ \mu $ of $ P_1 $ we define
  \begin{align*}
    s_\mu &= \mu_1, \\
    n_\mu &= |\{ \nu \in \NN^2 : \nu_1 = \mu_1, \nu_2 \le d-\nu_1, \nu \notin
      P_1 \}|, \\
    r_\mu &= d+1-s_\mu-n_\mu, \\
    l_\mu &= |\{ \nu \in \NN^2 : d-\nu_1-\nu_2 = n_\mu, \nu \in P_3 \}|.
  \end{align*}
\end{Notation}

As a sequel to \cite{wir}, we will now first restrict our attention to tropical
fan curves in $L_2^3$ and prove a new obstruction to relative realizability, in
the case of curves which are only contained in two opposite cones of $ L^3_2 $,
say $\cone([e_0],[e_3])$ and $\cone([e_1],[e_2])$. Afterwards, we show that
this necessary condition also results in sufficient conditions to relative
realizability.

\begin{Proposition} \label{prop:andreas}
  Let $C$ be a tropical fan curve of degree $d$ in $L_2^3$ contained only in
  the two opposite cones $\cone([e_0],[e_3])$ and $\cone([e_1],[e_2])$. If
  there is a vertex $ \mu $ of $ P_1 $ or $ P_3 $ such that $ l_\mu < r_\mu $,
  then $C$ is not relatively realizable.
\end{Proposition}

\begin{sidepic}{prop-4}
  \textit{Proof.}
  Without loss of generality, we may assume that $\mu$ is a ``lower vertex'' of
  $P_1$, \ie $(\mu_1,j) \notin P_1$ for all $j < \mu_2$ (otherwise
  permute the coordinates via $(0,1,2,3) \mapsto (3,1,2,0)$). We want to
  apply Lemma \ref{lemma:equpoly}. So choose $l = \mu_2+1$, $k = \mu_1$ and
  $m$ such that any lattice point $\nu$ in $P_3$ with $\nu_1 +
  \nu_2 = d - n_\mu$ is among the $ r_\mu-1 $ points $(d-n_\mu-m,m),\ldots,
  (d-n_\mu-m-r_\mu+2,m+r_\mu-2)$.
\end{sidepic}

  Let us assume that there is a polynomial $f \in K[x]$ realizing $C$, so that
  $\Delta_1 = P_1$ and $\Delta_3 = P_3$. Consider the equation from Lemma
  \ref{lemma:equpoly}. We see that any coefficient $a_\nu$ for $ \nu \in
  \Lambda_3 $ and any coefficient $b_\nu$ for $ \nu \in \Lambda_1 $, except
  $b_\mu$, that appears in this equation must be zero since its corresponding
  point is not contained in $\Delta_3$ or $\Delta_1$, respectively. Hence we
  must have $b_\mu = 0$. However, this is a contradiction to the fact that
  $\mu$ is a vertex of $\Delta_1$. This means that such a polynomial cannot
  exist, hence $C$ is not relatively realizable. \hfill $ \Box $

\vspace{1ex}

\begin{Remark}
  In the proof of Proposition \ref{prop:andreas} we do not use the complete
  information that $C$ is only contained in two opposite cones of $L_2^3$.
  Instead, with the notation of Lemma \ref{lemma:equpoly} we only use that $
  P_3 \cap \Lambda_3 = \emptyset $, and $ P_1 \cap \Lambda_1 $ is a single
  point, which is a vertex of $ P_1 $. So if these conditions are fulfilled
  for any curve $C$ in $ L^3_2 $, then $C$ is not relatively realizable.
\end{Remark}

We now show that for this special class of tropical fan curves in $L_2^3$,
the obstructions of Proposition \ref{prop:andreas} give in fact rise to
conditions that are equivalent to relative realizability.

\begin{Proposition} \label{prop:posreal}
  Let $C$ be a tropical fan curve of degree $d$ in $L_2^3$ only contained in
  the two opposite cones $\cone([e_1],[e_2])$ and $\cone([e_0],[e_3])$. Using
  Notation \ref{not-numbers}, the curve $C$ is relatively realizable if and
  only if $ l_\mu \geq r_\mu$ for all vertices $\mu$ of $P_3$ and $P_1$.
\end{Proposition}

\begin{proof}
  We already know by Proposition \ref{prop:andreas} that $C$ cannot be
  realizable if there is a vertex $\mu$ of $P_1$ or $P_3$ with $l_\mu < r_\mu$.
  So assume that we have $l_\mu \geq r_\mu$ for any vertex $\mu$ of $P_1$ and
  $P_3$. We will construct a polynomial $f^\mu = \sum_{|\nu| = d} a_\nu x^\nu
  \in K[x_0,x_1,x_2,x_3]$ for every vertex $\mu$ of $P_1$ (resp.\ $P_3$) such
  that $ \Newt(f^\mu_1) \subset P_1$ and $\Newt(f^\mu_3) \subset P_3$, and the
  Newton polytope $ \Newt(f^\mu_1) $ (resp.\ $ \Newt(f^\mu_3) $) contains the
  point $ \mu $. Summing up these polynomials with generic coefficients, we
  then get a polynomial $f$ such that $\Newt(f_3) = P_3$ and $\Newt(f_1) =
  P_1$.

  Let $\mu$ be a vertex of $P_3$. We set
  $f^\mu = x_0^{a_0}x_1^{a_1}x_2^{a_2}x_3^{a_3}(x_0+x_3)^a$ with
  \begin{itemize}
  \item $a = d-s_\mu-n_\mu$, \ie the Newton polytope of $(x_0+x_3)^a$ contains
    exactly $r_\mu$ points,
  \item $a_1$ and $a_2$ as the number of points on the diagonal
    $\conv((s_\mu,0),(0,s_\mu))$ on the left and right side of $P_3$, \ie we
    have $a_1+a_2 = n_\mu$,
  \item by assumption we can choose a part of $\conv((n_\mu,0),(n_\mu,d-n_\mu))
    \cap P_1$ of lattice length $a$, \ie it contains $ r_\mu $ points. Let
    $a_0$ and $a_3$ be the number of lattice points on the left and right side
    of this part, \ie we have $a + a_0 + a_3 = d - n_\mu$.
  \end{itemize}
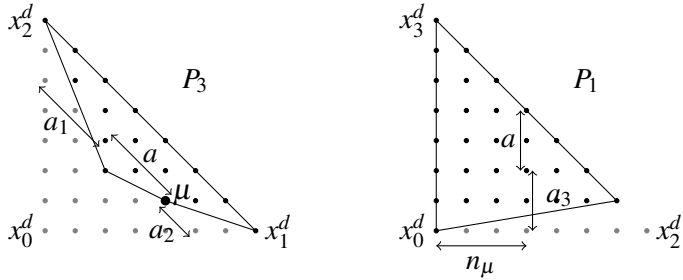
  \begin{center} \begin{tikzpicture}[scale=0.4]
   \filldraw[color=gray] (0,0) circle (2pt);
   \filldraw[color=gray] (1,0) circle (2pt);
   \filldraw[color=gray] (2,0) circle (2pt);
   \filldraw[color=gray] (3,0) circle (2pt);
   \filldraw[color=gray] (4,0) circle (2pt);
   \filldraw[color=gray] (5,0) circle (2pt);
   \filldraw[color=gray] (6,0) circle (2pt);
   \filldraw (7,0) circle (2pt);
   \filldraw[color=gray] (0,1) circle (2pt);
   \filldraw[color=gray] (1,1) circle (2pt);
   \filldraw[color=gray] (2,1) circle (2pt);
   \filldraw[color=gray] (3,1) circle (2pt);
   \filldraw (4,1) circle (4pt);
   \filldraw (5,1) circle (2pt);
   \filldraw (6,1) circle (2pt);
   \filldraw[color=gray] (0,2) circle (2pt);
   \filldraw[color=gray] (1,2) circle (2pt);
   \filldraw (2,2) circle (2pt);
   \filldraw (3,2) circle (2pt);
   \filldraw (4,2) circle (2pt);
   \filldraw (5,2) circle (2pt);
   \filldraw[color=gray] (0,3) circle (2pt);
   \filldraw[color=gray] (1,3) circle (2pt);
   \filldraw (2,3) circle (2pt);
   \filldraw (3,3) circle (2pt);
   \filldraw (4,3) circle (2pt);
   \filldraw[color=gray] (0,4) circle (2pt);
   \filldraw[color=gray] (1,4) circle (2pt);
   \filldraw (2,4) circle (2pt);
   \filldraw (3,4) circle (2pt);
   \filldraw[color=gray] (0,5) circle (2pt);
   \filldraw (1,5) circle (2pt);
   \filldraw (2,5) circle (2pt);
   \filldraw[color=gray] (0,6) circle (2pt);
   \filldraw (1,6) circle (2pt);
   \filldraw (0,7) circle (2pt);

   \draw (0,7) -- (7,0) -- (4,1) -- (2,2) -- (0,7);
   \draw[<->] (-0.2,4.8) -- (1.8,2.8);
   \draw[<->] (4.2,1.2) -- (2.2,3.2);
   \draw[<->] (4.8,-0.2) -- (3.8,0.8);
   \draw (3.5,2.5) node {$a$};
   \draw (0.4,3.5) node {$a_1$};
   \draw (3.9,-0.1) node {$a_2$};
   \draw (0,7) node[anchor = east] {$x_2^d$};
   \draw (7,0) node[anchor = west] {$x_1^d$};
   \draw (0,0) node[anchor = east] {$x_0^d$};
   \draw (3.9,1.2) node[anchor = west] {$\mu$};
   \draw (5,5) node {$ P_3 $};

   \begin{scope}[xshift = 3cm]
   \filldraw (10,0) circle (2pt);
   \filldraw[color=gray] (11,0) circle (2pt);
   \filldraw[color=gray] (12,0) circle (2pt);
   \filldraw[color=gray] (13,0) circle (2pt);
   \filldraw[color=gray] (14,0) circle (2pt);
   \filldraw[color=gray] (15,0) circle (2pt);
   \filldraw[color=gray] (16,0) circle (2pt);
   \filldraw[color=gray] (17,0) circle (2pt);
   \filldraw (10,1) circle (2pt);
   \filldraw (11,1) circle (2pt);
   \filldraw (12,1) circle (2pt);
   \filldraw (13,1) circle (2pt);
   \filldraw (14,1) circle (2pt);
   \filldraw (15,1) circle (2pt);
   \filldraw (16,1) circle (2pt);
   \filldraw (10,2) circle (2pt);
   \filldraw (11,2) circle (2pt);
   \filldraw (12,2) circle (2pt);
   \filldraw (13,2) circle (2pt);
   \filldraw (14,2) circle (2pt);
   \filldraw (15,2) circle (2pt);
   \filldraw (10,3) circle (2pt);
   \filldraw (11,3) circle (2pt);
   \filldraw (12,3) circle (2pt);
   \filldraw (13,3) circle (2pt);
   \filldraw (14,3) circle (2pt);
   \filldraw (10,4) circle (2pt);
   \filldraw (11,4) circle (2pt);
   \filldraw (12,4) circle (2pt);
   \filldraw (13,4) circle (2pt);
   \filldraw (10,5) circle (2pt);
   \filldraw (11,5) circle (2pt);
   \filldraw (12,5) circle (2pt);
   \filldraw (10,6) circle (2pt);
   \filldraw (11,6) circle (2pt);
   \filldraw (10,7) circle (2pt);

   \draw (10,0) node[anchor = east] {$x_0^d$};
   \draw (17,0) node[anchor = west] {$x_2^d$};
   \draw (10,7) node[anchor = east] {$x_3^d$};
   \draw (10,0) -- (16,1) -- (10,7) -- (10,0);
   \draw[<->] (13.2,0) -- (13.2,2);
   \draw[<->] (12.8,2) -- (12.8,4);
   \draw[<->] (10,-0.5) -- (13,-0.5);
   \draw (11.5,-1.2) node {$n_\mu$};
   \draw (12.4,3) node {$a$};
   \draw (14.1,1.1) node {$a_3$};
   \draw (15,5) node {$ P_1 $};
   \end{scope}
  \end{tikzpicture} \end{center}
  Then
  \begin{align*}
    f^\mu_3 &= x_0^{a_0}x_1^{a_1}x_2^{a_2}(-x_0-x_1-x_2)^{a_3}(-x_1-x_2)^a \\
    \text{and} \quad
    f^\mu_1 &= x_0^{a_0}(-x_0-x_2-x_3)^{a_1}x_2^{a_2}x_3^{a_3}(x_0+x_3)^a,
  \end{align*}
  so the Newton polytope of $f^\mu_3$ is contained in $P_3$, and the Newton
  polytope of $f^\mu_1$ is contained in $P_1$. Moreover, the coefficient of $
  x^\mu $ in $ f^\mu_3 $ is $ \pm 1 $. Analogously, we define the polynomial
  $f^\mu$ for any vertex $\mu$ of $P_1$.

  Let now $f$ be a generic sum of these polynomials $f^\mu$, where $\mu$ ranges
  over all vertices of $P_3$ and $P_1$. Then $ \Newt(f_3) = P_3 $ and $
  \Newt(f_1) = P_1 $ by construction, which means that $ \Trop(f_3) = C_3 $ and
  $ \Trop(f_1) = C_1 $.

  Note that the projection $ p^{(0,1,2)} $ is injective on the cones $
  \cone([e_0],[e_1]) $, $ \cone([e_0],[e_2])$, and $\cone([e_1],[e_2]) $,
  whereas $ p^{(0,2,3)} $ is injective on the cones $ \cone([e_0],[e_2]) $,
  $ \cone([e_0],[e_3])$, and $\cone([e_2],[e_3]) $. So these two projections
  allow to reconstruct each part of $C$ which is not contained in the interior
  of $ \cone([e_1],[e_3]) $. But $C$ does not contain any rays in this interior
  by assumption, and $ \Trop (L+(f)) $ does not contain any rays in this region
  either: $(d,0)$ and $(0,d)$ are vertices of $P_3$, so $f$ contains a term
  $ f^{(d,0)} = (x_0+x_3)^d $, which implies that $ f_{(0,1,3)} $ has a
  non-zero term $ x_0^d $.

  So altogether we conclude that $ \Trop(f) = C $, as required.
\end{proof}

We will now address ourselves to the relative realizability of tropical curves
in $L_2^3$ which are not necessarily fan curves. We start with generalizing
Notation \ref{not-numbers}.

\begin{sidepic}{pic-1a}
\begin{Notation} \label{not-gnumbers}
  Let $P$ be a lattice polytope in $\mathbb R^2$, and let $\mu$ be a vertex of
  $P$. We define
  \begin{align*}
    s_\mu(P) &= \max_{\nu \in P}\; (\nu_1 + \nu_2) - (\mu_1 + \mu_2)  
      \;\; \text{(row number of $ \mu $ in $P$)}, \\
    n_\mu(P) &= |\{ \nu  \in \mathbb N^2: \nu_1 + \nu_2 = \mu_1 + \mu_2,
      \nu \notin P \}| \\
    &\qquad \text{(lattice points not in $P$ in the diagonal of $ \mu $)}, \\
    u(P) &= \min_{\nu \in P} \; \nu_1 \;\;
      \text{(lattice distance of $P$ to the vertical axis)}.
  \end{align*}
\end{Notation}
\end{sidepic}

\begin{Proposition} \label{prop:length}
  Let $C$ be a tropical curve of degree $d$ in $ L^3_2 $ such that $\star_C(0)$
  contains a classical line, and assume that this classical line is adjacent to
  the two vertices $[0,q,q,0]$ and $[q',0,0,q']$ with $ q,q' \in \RR_{>0}$.
  Moreover, let $ \mu=(\mu_1,\mu_2) $ be a vertex of the polytope $P$ in
  $\Subdiv(C_3)$ dual to $[0,q,q]$, and let $Q$ be the polytope in
  $\Subdiv(C_1)$ dual to $[q',0,q']$.
 
  Then the curve $C$ can only be relatively realizable if $(n_\mu(P)-u(Q)) \cdot
  q' \le s_\mu(P) \cdot q $.
\end{Proposition}

\begin{proof}
  To simplify the notation, let $s = s_\mu(P)$ and $n = n_\mu(P)$.
  Assume that $C$ is relatively realizable, say by $f = \sum_{|\nu| = d}
  a_\nu x^\nu \in \Ct[x_0,x_1,x_2]$. Due to the form of $C$ we know 
  that $P$ has a face $\tau$ parallel to $\conv((d,0),(0,d))$ corresponding to 
  the edge of $C$ through the origin. Let $ \eta,\xi \in \mathbb N^2$ such
  that $\tau = \conv(\eta,\xi)$. By Remark \ref{rem:initialval}, we can
  assume that $\val(a_\eta) = 0$. Note that we then have $\val(a_\nu) \geq
  0$ for any lattice point $\nu$ in $\Newt(f)$: the origin in $\mathbb R^2$
  is either a vertex of $C_3$ or contained in an edge of $C_3$. In any case,
  the polyhedron $\tau$ is a part of the polyhedron in the Newton subdivision
  of $C_3$ dual to the origin. With Remark \ref{rem:condsub}, we thus see that
  $\val(a_\xi) = 0$ and $\val(a_\nu) \geq 0$ for all lattice points $\nu$
  in $\Newt(f)$. Let $A \subset \mathbb N^2$ such that $(A,(q,q)) \in
  \Subdiv(C_3)$, \ie we have $P = \conv(A)$ and $\mu \in A$. Due to the
  conditions on $f$ coming from $C$, we then know that $ \val(a_\nu) =
  (m+s-(\nu_1+\nu_2)) \cdot q $ for all $\nu=(\nu_1,\nu_2) \in A$ and
  $\val(a_\nu) > (m+s-(\nu_1+\nu_2))\cdot q$ for all $\nu \notin P$, where $m
  = \mu_1+\mu_2$. Hence $ \val(a_\mu) = sq $, and $\val(a_\nu) > s q $ for all
  $\nu$ with $\nu_1+\nu_2 < m$, or with $\nu_1 + \nu_2 = m$ and $\nu \notin P$.

  Let us now consider the polynomial $g = \sum_{|\nu| = d} a_\nu^{sq} x^\nu \in
  K[x_0,x_1,x_2]$, \cf Notation \ref{notation:rhs}, \ie the polynomial
  consisting of only those terms of $f$ corresponding to $t^{sq}$. With the
  given arguments, we see that $\mu$ is a vertex of $\Newt(g)$ with
  $n_\mu(\Newt(g)) \geq n$. For simplicity, let $n_\mu(g) = n_\mu(\Newt(g))$. 

  Let us consider $g_1 = g_{(0,2,3)}$. We want to apply Proposition
  \ref{prop:andreas} to the polynomial $g$ (in fact to its tropicalization
  $\Trop(L+(g))$), hence, let us consider the $n_\mu(g)$-th row of
  $\Newt(g_1)$. First of all, note that if we have $f_1(x) = f_{(0,2,3)}(x) =
  \sum_{|\nu| = d} b_\nu x^\nu$ as in Notation \ref{not:13}, we see that $g_1 =
  \sum_{|\nu| = d} b_\nu^{sq} x^\nu$. As above, we know that $Q$ has a face
  $\sigma$ parallel to $\conv((0,0),(0,d))$ which is contained in the part of
  the Newton subdivision of $C_1$ dual to the origin. Hence, if we write
  $\sigma = \conv(\eta,\xi)$ then $ \val(b_\eta) = \val(b_\xi) $ is minimal in
  $\Newt(f_1)$. Since the coefficients $a_\nu$ are linear combinations of the
  $b_\nu$ and vice versa, this minimal valuation must be zero. So the
  conditions on $\Subdiv(C_1)$ ensure that $\val(b_\nu) \geq (\nu_1-u) \cdot
  q'$ for all lattice points $\nu$ in $\Newt(f_1)$. If $\nu$ is a lattice point
  in $\Newt(f_1)$ with $\nu_1 q' > uq' + sq $, we thus see that $ \val(b_\nu) >
  sq $, and so $\nu \notin \Newt(g_1)$. 

  Hence, if we had $ n_\mu(g) q' > uq' + sq $, there would be no points
  in the $n_\mu(g)$-th diagonal of $\Newt(g_1)$. By Proposition
  \ref{prop:andreas}, we know that such a polynomial $g$ cannot exist, hence
  $C$ would not be relatively realizable. Thus, we must have $n_\mu(g) q'
  \le uq' + sq $, and we finally get $ nq' \le n_\mu(g) \, q' \le uq' + sq $,
  \ie $ (n-u) \, q' \leq sq $.
\end{proof}

From now on, we will restrict our attention to tropical curves $C$ in $ L^3_2 $
with only one bounded edge, which in addition passes through the origin. In
this case, as already presented in the introduction, we will be able to
give a complete non-algorithmic answer to the relative realizability problem.

Note that the subdivisions of the two Newton polytopes $ P_3 $ and $ P_1 $ 
of the projections $ C_3 $ and $ C_1 $ are trivial in this case. Moreover, the 
numbers of Notation \ref{not-gnumbers} adapted to this situation, as well
as to vertices of $P_1$ instead of $P_3$, are exactly the numbers defined in
Notation \ref{not-numbers}.

\begin{Corollary} \label{cor:length}
  Let $C$ be a tropical curve of degree $d$ in $L_2^3$ that has only one
  bounded edge, passing through the origin with adjacent vertices $[0,q,q,0]$
  and $[q',0,0,q']$ for some $ q,q' \in \RR_{>0} $. Then $C$ can only be
  relatively realizable if $ n_\mu \cdot q' \leq s_\mu \cdot q $ for every
  vertex $ \mu $ of $ P_3 $ and $ n_\mu \cdot q \leq s_\mu \cdot q' $ for every
  vertex $ \mu $ of $ P_1 $.
\end{Corollary}

\begin{proof}
  The statement for a vertex $ \mu $ of $ P_3 $ follows from Proposition
  \ref{prop:length} with $ P=P_3 $, $ Q=P_1 $, and $u(Q) = 0$. The case $ \mu
  \in P_1 $ is analogous.
\end{proof}

\begin{sidepic}{curvel32}
  \begin{Example}
    For $ q,q' \in \RR_{>0}$, consider the tropical curve $C$ in $L_2^3$ shown
    on the right. It has vertices $[0,q,q,0]$ and $[q',0,0,q']$, the joining
    edge has weight $3$, and the cones of its recession fan are spanned by
    the vectors $ [1,0,0,0] $ (with weight $2$), $ [1,0,0,3] $, $ [0,1,0,0] $,
    and $ [0,2,3,0] $.

    The corresponding Newton polytopes $ P_3 $ and $ P_1 $ are shown in the
    picture below.
  \end{Example}
\end{sidepic}

  
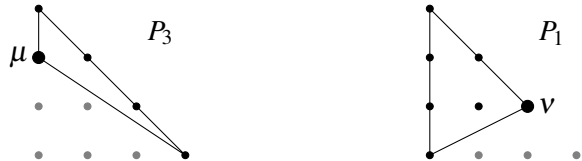
\begin{center} \begin{tikzpicture}[scale=0.65]
   \filldraw[color=gray] (0,0) circle (2pt);
   \filldraw[color=gray] (1,0) circle (2pt);
   \filldraw[color=gray] (0,1) circle (2pt);
   \filldraw[color=gray] (2,0) circle (2pt);
   \filldraw[color=gray] (1,1) circle (2pt);
   \filldraw (0,2) circle (3.5pt);
   \filldraw (3,0) circle (2pt);
   \filldraw (0,3) circle (2pt);
   \filldraw (2,1) circle (2pt);
   \filldraw (1,2) circle (2pt);
   \draw (0,3) -- (3,0) -- (0,2) -- (0,3);
   \draw (2.5,2.5) node {$P_3$};
   \draw (-0.4,2) node {$\mu$};

   \filldraw (8,0) circle (2pt);
   \filldraw[color=gray] (9,0) circle (2pt);
   \filldraw (8,1) circle (2pt);
   \filldraw (8,2) circle (2pt);
   \filldraw (9,1) circle (2pt);
   \filldraw[color=gray] (10,0) circle (2pt);
   \filldraw[color=gray] (11,0) circle (2pt);
   \filldraw (8,3) circle (2pt);
   \filldraw (10,1) circle (3.5pt);
   \filldraw (9,2) circle (2pt);
   \draw (8,0) -- (10,1) -- (8,3) -- (8,0);
   \draw (10.5,2.5) node {$P_1$};
   \draw (10.4,1) node {$\nu$};
  \end{tikzpicture} \end{center}

  The only interesting vertices are $\mu$ in $P_3$ and $\nu$ 
  in $P_1$, as for all other vertices $ \eta $ of $ P_3 $ or $ P_1 $ we have
  $ s_\eta = n_\eta = 0 $. As $ n_\mu = s_\nu = 2 $ and $ n_\nu = s_\mu =
  1 $, Corollary \ref{cor:length} implies that $C$ can only be realizable
  if $ 2 q' \le q $ and $ q \le 2 q' $, \ie if $ \frac q{q'} = 2 $. On the
  other hand, it can be checked using Algorithm \ref{alg-1} (or Proposition
  \ref{prop:posreal2} below) that $C$ is in fact realizable if $ \frac q{q'} =
  2 $.

\begin{Remark} \label{rem:length}
  Let us consider the limit case $ q=0 $, \ie if the star of $C$ at the origin
  contains $ \cone ([1,0,0,1]) $ with some weight, but not necessarily $
  \cone([0,1,1,0]) $, and $P$ is considered to be the polyhedron in the
  subdivision of $ C_3 $ dual to the origin. If we follow the proof of
  Proposition \ref{prop:length}, we would choose $g = \sum_{|\nu| = d}
  a_\nu^{0} x^\nu$ in $K[x_0,x_1,x_2]$ for any vertex of $P$ since we have $q =
  0$. So the Newton polytope of $g$ completely contains $P$. Hence, in the
  case $q = 0$, the proof of Proposition \ref{prop:length} only works if we
  choose $\mu$ to be a vertex of $P$ such that $s_\mu(P)$ is maximal among the
  vertices of $P$. The resulting necessary condition for realizability is $
  n_\mu(P) \le u(Q) $. In the other limit case $ q'=0 $, the statement of
  Proposition \ref{prop:length} is trivial.
  
  Correspondingly, for Corollary \ref{cor:length} let us consider curves with
  one bounded edge, passing through the origin with direction $ [1,0,0,1] $,
  such that the origin is one of its adjacent vertices (and thus $ q=0 $ or $
  q'=0 $), and the rays adjacent to the two vertices are contained in two
  opposite cones of $ L^3_2 $. In this case, the necessary condition for
  realizability is $n_\mu = 0$, where $\mu$ is a vertex of $P_3$, or $P_1$
  respectively, such that $s_\mu$ is maximal. Due to the form of $P_3$, we see
  that if $n_\mu = 0$ for a vertex of $P_3$ with maximal $s_\mu$, then we have
  $n_\nu = 0$ for any vertex $\nu$ of $P_3$. The analogous statement is true
  for $P_1$. Hence, in the degenerated case $q =0$, or $q' = 0$ the statement
  of Corollary \ref{cor:length} still holds.
\end{Remark}

Let us now prove that tropical curves in $L_2^3$ with one bounded edge, passing
through the origin, and with realizable recession fan, are realizable if and
only if they fulfill the length conditions of Corollary \ref{cor:length}. The
proof is based on the same idea as the proof of Proposition \ref{prop:posreal}.
Note that the realizability of the recession fan can be checked by Proposition
\ref{prop:posreal} since this fan is also contained in two opposite cones of
$L_2^3$.
 
\begin{Proposition} \label{prop:posreal2}
  Let $C$ be a tropical curve of degree $d$ in $L_2^3$ with exactly one bounded
  edge, passing through the origin, and with vertices $ [0,q,q,0] $ and $
  [q',0,0,q'] $. Assume moreover that the recession fan of $C$ is relatively
  realizable. Then $C$ is relatively realizable in $L$ if and only if
  $ n_\mu \cdot q' \leq s_\mu \cdot q $ for every vertex $ \mu $ of $ P_3 $ and
  $ n_\mu \cdot q \leq s_\mu \cdot q' $ for every vertex $ \mu $ of $ P_1 $.
\end{Proposition}

\begin{proof}
  We know already by Corollary \ref{cor:length} that $C$ is not relatively 
  realizable if one of the length conditions is not fulfilled. So let us prove
  the other direction of the statement.

  With Remark \ref{remark:decomp}, the conditions on the coefficients $a_\nu$
  of a polynomial $ f = \sum_\nu a_\nu x^\nu \in \Ct[x_0,x_1,x_2] $ realizing
  $C$ can be decomposed into conditions on the parameters $a_\nu^k$ for fixed
  $k \in \RHS(C)$, where $ a_\nu = \sum_{k \in \RHS(C)} a_\nu^k t^k $. Hence if
  we find a polynomial $ f^k = \sum_\nu a_\nu^k x^\nu $ fulfilling the
  conditions on the $a_\nu^k$ for all $k \in \RHS(C)$, then $f =
  \sum_{k\in\RHS(C)} f^k t^k$ realizes $C$.

  Since $C$ has only two vertices, we have $\Subdiv(C_3) = \{ (P_3,(q,q)) \}$
  and $\Subdiv(C_1) = \{ (P_1,(-q',0)) \}$. Assuming that $\val (a_{(d,0)}) =
  0$, the conditions on the $a_\nu^k$ have the following form: For any lattice
  point $\nu$ in $P_3$ we get the condition $a_\nu^k = 0$ for all $k \in
  \RHS(C)$ with $k < s_\nu q$. For any vertex $\nu$ of $P_3$, we get the
  additional condition $a_\nu^{s_\nu q} \neq 0$. Correspondingly, using
  Notation \ref{not:13} the conditions on the coefficients $ b_\nu $ of $ f_1 $
  are given by $b_\nu^k = 0$ for all lattice points $\nu$ in $P_1$ and all $k
  \in \RHS(C)$ with $k < s_\nu q'$, and $b_\nu^{s_\nu q'} \neq 0$ for all
  vertices $\nu$ of $P_1$.  

  For any vertex $\mu$ of $P_3$ or $P_1$, let $f^\mu$ be as defined in the
  proof of Proposition \ref{prop:posreal}. For $k \in \RHS(C)$, we define $f^k$
  as a generic sum over $f^\mu$, where $\mu$ ranges over all vertices $\mu$ of
  $P_3$ with $ s_\mu q= k$ and all vertices $\mu$ of $P_1$ with $ s_\mu q'
  = k$. As the recession fan of $C$ is relatively realizable and projects to
  curves dual to the same polytopes $P_1$ and $P_3$, we know by Proposition
  \ref{prop:posreal} that $l_\mu \geq r_\mu$ for any vertex $\mu$ of $P_1$ and
  $P_3$. Hence, the polynomials $f^k_3$ and $f^k_1$ only contain monomials
  $x^\nu$ with $\nu \in P_3$ and $\nu \in P_1$, respectively. We claim that
  $f^k$ fulfills the conditions on the parameters $a_\nu^k$: By construction,
  we have $a_\nu^{s_\nu q} \neq 0$ for all vertices $\nu$ of $P_3$ and
  $b_\nu^{s_\nu q'} \neq 0$ for all vertices $\nu$ of $P_1$. To show that
  $a_\nu^k = 0$ for all lattice points in $P_3$ with $k < s_\nu q$ and $b_\nu^k
  = 0$ for all lattice points in $P_1$ with $k < s_\nu q'$, we show that these
  conditions are fulfilled for any polynomial $f^\mu$ arising in the definition
  of $f^k$.

  Let $\mu$ be a vertex of $P_3$ with $ s_\mu q = k$. Due to the definition of
  $f^\mu$, it is clear that $f^\mu$ only contains monomials $x^\nu$ with $
  s_\nu \le s_\mu $, and thus it fulfills the conditions $a_\nu^k = 0$ for $
  s_\nu q > k$. So let us consider $f^\mu_1$. We have seen in the proof of
  Proposition \ref{prop:posreal} that $f^\mu_1$ only contains monomials $x^\nu$
  with $ s_\nu \leq n_\mu $. But we have assumed $ n_\mu q' \le s_\mu q = k$,
  hence the polynomial $f^\mu$ only contains monomials with $ s_\nu q' \le
  n_\mu q' \le k$, \ie the condition $b_\nu^k = 0$ is fulfilled for all $\nu$
  with $ s_\nu q' > k$.

  Analogously, one shows that $f^\mu$ with $\mu$ a vertex of $P_1$ and $ s_\mu
  q' = k$ fulfills the above conditions. Setting $f = \sum_{k \in
  \RHS(C)} f^k t^k$, we thus have $\Subdiv(f_3) = \Subdiv(C_3)$ and
  $\Subdiv(f_1) = \Subdiv(C_1)$. As in the proof of Proposition
  \ref{prop:posreal}, we see that $\Trop(f)$ does not meet the relative
  interior of $\cone([e_1],[e_3])$, so $f$ realizes $C$.
\end{proof}

\begin{Remark} \label{rem:degenerate}
  In the spirit of Remark \ref{rem:length}, let us consider the degenerated
  cases when $ q=0 $ or $ q'=0 $. If $ q=q'=0 $ then $C$ has the origin as its
  only vertex, so we can use Proposition \ref{prop:posreal} to check if $C$ is
  relatively realizable. If only $ q=0 $, the curve has still one bounded edge,
  of direction $ [1,0,0,1] $ and weight $d$, with one vertex being the origin.
  The necessary conditions in this case are $n_\mu = 0$ for all vertices $\mu$
  of $P_3$. Following the proof of Proposition \ref{prop:posreal2} (with the
  assumption that $\Rec(C)$ is realizable), we see that these conditions are 
  also sufficient.
\end{Remark}

\begin{Remark}[Result of Proposition \ref{prop:posreal2}]
  Let $C$ be a tropical curve as above, of degree $d$ in $L_2^3$ with exactly
  one bounded edge, passing through the origin, and with vertices $ [0,q,q,0] $
  and $ [q',0,0,q'] $. We also allow the degenerated cases $q = 0$ and/or $q' =
  0$. Assume that $\Rec(C)$ is relative realizable.

  Consider the case that there are vertices $\mu,\nu$ of $P_3$ and $P_1$,
  respectively, such that $n_\mu \neq 0$ and $n_\nu \neq 0$. Note that we also
  have $s_\mu \neq 0$ and $s_\nu \neq 0$ since the $0$-th row is completely
  contained in $P_3$ and $P_1$. Among the vertices $\eta$ of $P_3$ with $s_\eta
  \neq 0$, let $\mu$ be a vertex such that $\frac{n_\mu}{s_\mu}$ is maximal.
  Similarly, let $\nu$ be a vertex of $P_1$ with $n_\nu \neq 0$ such that
  $\frac{s_\nu}{n_\nu}$ is minimal. Using Proposition \ref{prop:posreal2}, we
  see that $C$ is relatively realizable if and only if
    $$ q = q' = 0 \quad\text{or}\quad
       \frac q{q'} \in \left[ \frac{n_\mu}{s_\mu},
       \frac {s_\nu}{n_\nu} \right]. $$
  Note that if $q \neq 0$ (resp.\ $q' \neq 0$) and $C$ is relatively
  realizable, then by the length conditions of Proposition \ref{prop:posreal2}
  we must also have $q' \neq 0$ (resp.\ $q \neq 0$).

  If we have $s_\mu = 0$ for all vertices $\mu$ of $P_1$, with the argument
  from above we also have $n_\mu = 0$ for all vertices $\mu$ of $P_1$.
  Moreover, using Proposition \ref{prop:andreas}, we also have $n_\mu = 0$ for
  all vertices $\mu$ of $P_3$. The same is true if we have $s_\mu = 0$ for all
  vertices $\mu$ of $P_3$ (swap the arguments).  In any of these two cases, the
  tropical curve $C$ is relatively realizable for any $q,q' \in \mathbb R_{\geq
  0}$ by Proposition \ref{prop:posreal2} with Remark \ref{rem:degenerate}.

  Finally, let us consider the case that $n_\mu = 0$ for all vertices $\mu$ of
  $P_1$. In this case, the conditions $n_\mu \cdot q \le \mu_1 \cdot q' $ are
  trivial. Hence, by Proposition \ref{prop:posreal2} with Remark
  \ref{rem:degenerate} the tropical curve $C$ is relatively realizable for all
  $q,q'$ with either $q' = 0$ or $\frac q{q'} \in \left[ \frac{n_\mu}{s_\mu},
  \infty\right)$, where $\mu$ is a vertex of $P_3$ with $ s_\mu \neq 0 $ such
  that $\frac{n_\mu}{s_\mu}$ is maximal. The analogous statement is true if we
  have $n_\mu = 0$ for all vertices $\mu$ of $P_3$. If we have $n_\mu = 0$ for
  all vertices $\mu$ of $P_3$ and $P_1$, then $C$ is relatively realizable for
  all $q, q' \in \mathbb R_{\geq 0}$.
\end{Remark}

Let us now return to the case of a tropical curve $C$ having vertices $\mu$ of
$P_3$ and $\nu$ of $P_1$ such that $n_\mu \neq 0$ and $n_\nu \neq 0$. Assume
that $\mu$ is chosen such that $\frac{n_\mu}{s_\mu}$ maximal and $\nu$ is
chosen such that $\frac{s_\nu}{n_\nu}$ is minimal. Finally, let $I = \left[
\frac{n_\mu}{s_\mu}, \frac {s_\nu}{n_\nu} \right]$.

For $q \neq 0$ and $q' \neq 0$, we know by Proposition \ref{prop:posreal2} that
the tropical curve $C$ with vertices $[0,q,q,0]$ and $[q',0,0,q']$ is
relatively realizable if and only if its recession fan is relatively realizable
and $\frac q{q'} \in I$. However, it may happen that $I$ is empty. In the
following two examples, we will see that the relative realizability of the
recession fan does not correlate with $I$ being non-empty.

\begin{Example} \label{ex-rec}
  For $q,q' \in \mathbb R_{>0}$, consider the tropical curve $C$ in $L_2^3$
  which has only one bounded edge of weight $4$ with adjacent vertices
  $[0,q,q,0]$ and $[q',0,0,q']$, and whose projections have the following
  Newton polytopes:
  
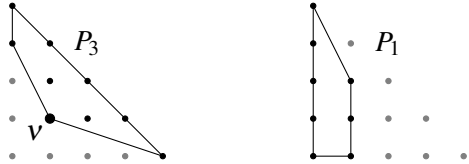
\begin{center}
    \begin{tikzpicture}[scale = 0.5]
      \filldraw (0,0)[color = gray] circle (2pt);
      \filldraw (0,1)[color = gray] circle (2pt);
      \filldraw (0,2)[color = gray] circle (2pt);
      \filldraw (0,3) circle (2pt);
      \filldraw (0,4) circle (2pt);
      \filldraw (1,0)[color = gray] circle (2pt);
      \filldraw (1,1) circle (3.5pt);
      \filldraw (1,2) circle (2pt);
      \filldraw (1,3) circle (2pt);
      \filldraw (2,0)[color = gray] circle (2pt);
      \filldraw (2,1) circle (2pt);
      \filldraw (2,2) circle (2pt);
      \filldraw (3,0)[color = gray] circle (2pt);
      \filldraw (3,1) circle (2pt);
      \filldraw (4,0) circle (2pt);
      \draw (0,4) -- (0,3) -- (1,1) -- (4,0) -- (0,4);
      \draw (2,3) node {$P_3$};
      \draw (1.1,1.1) node[anchor = north east] {$\nu$};
    \begin{scope}[xshift = 8cm]
      \filldraw (0,0) circle (2pt);
      \filldraw (0,1) circle (2pt);
      \filldraw (0,2) circle (2pt);
      \filldraw (0,3) circle (2pt);
      \filldraw (0,4) circle (2pt);
      \filldraw (1,0) circle (2pt);
      \filldraw (1,1) circle (2pt);
      \filldraw (1,2) circle (2pt);
      \filldraw[color = gray] (1,3) circle (2pt);
      \filldraw[color = gray] (2,0) circle (2pt);
      \filldraw[color = gray] (2,1) circle (2pt);
      \filldraw[color = gray] (2,2) circle (2pt);
      \filldraw[color = gray] (3,0) circle (2pt);
      \filldraw[color = gray] (3,1) circle (2pt);
      \filldraw[color = gray] (4,0) circle (2pt);
      \draw (0,0) -- (1,0) -- (1,2) -- (0,4) -- (0,0);
      \draw (2,3) node {$P_1$};
    \end{scope}
    \end{tikzpicture}
  \end{center}
  For the vertex $\nu$ we have $n_\nu = 2$ and $r_\nu = 1 > 0 = l_\nu$. So the
  recession fan of $C$ is not relatively realizable by Proposition
  \ref{prop:andreas}, and thus $C$ is not relatively realizable either.
  However, for any vertex $\mu$ of $P_3$ with $s_\mu \neq 0$ we have $
  \frac{n_\mu}{s_\mu} = 1$, and for any vertex $\eta$ of $P_1$ with $n_\eta
  \neq 0$ we have $\frac{s_\eta}{n_\eta} = 1$. Hence $ I = \Big[
  \frac{n_\mu}{s_\mu}, \frac{s_\eta}{n_\eta} \Big] = [1,1] $ is non-empty.
\end{Example}

\begin{Example} \label{ex-empty}
  For $q,q' \in \mathbb R_{>0}$ consider the tropical curve $C$ in $L_2^3$
  which has only one bounded edge of weight $5$ with adjacent vertices
  $[0,q,q,0]$ and $[q',0,0,q']$, and whose projections are described by the
  following Newton polytopes:
  
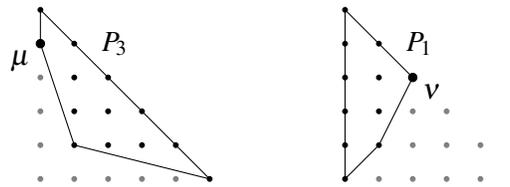
\begin{center}
    \begin{tikzpicture}[scale = 0.45]
      \filldraw (0,0)[color = gray] circle (2pt);
      \filldraw (0,1)[color = gray] circle (2pt);
      \filldraw (0,2)[color = gray] circle (2pt);
      \filldraw (0,3)[color = gray] circle (2pt);
      \filldraw (0,4) circle (3.5pt);
      \filldraw (0,5) circle (2pt);
      \filldraw (1,0)[color = gray] circle (2pt);
      \filldraw (1,1) circle (2pt);
      \filldraw (1,2) circle (2pt);
      \filldraw (1,3) circle (2pt);
      \filldraw (1,4) circle (2pt);
      \filldraw (2,0)[color = gray] circle (2pt);
      \filldraw (2,1) circle (2pt);
      \filldraw (2,2) circle (2pt);
      \filldraw (2,3) circle (2pt);
      \filldraw (3,0)[color = gray] circle (2pt);
      \filldraw (3,1) circle (2pt);
      \filldraw (3,2) circle (2pt);
      \filldraw (4,0)[color = gray] circle (2pt);
      \filldraw (4,1) circle (2pt);
      \filldraw (5,0) circle (2pt);
      \draw (0,5) -- (0,4) -- (1,1) -- (5,0) -- (0,5);
      \draw (2.2,4) node {$P_3$};
      \draw (0,4.1) node[anchor = north east] {$\mu$};
    \begin{scope}[xshift = 9cm]
      \filldraw (0,0) circle (2pt);
      \filldraw (0,1) circle (2pt);
      \filldraw (0,2) circle (2pt);
      \filldraw (0,3) circle (2pt);
      \filldraw (0,4) circle (2pt);
      \filldraw (0,5) circle (2pt);
      \filldraw (1,0)[color = gray] circle (2pt);
      \filldraw (1,1) circle (2pt);
      \filldraw (1,2) circle (2pt);
      \filldraw (1,3) circle (2pt);
      \filldraw (1,4) circle (2pt);
      \filldraw (2,0)[color = gray] circle (2pt);
      \filldraw (2,1)[color = gray] circle (2pt);
      \filldraw (2,2)[color = gray] circle (2pt);
      \filldraw (2,3) circle (3.5pt);
      \filldraw (3,0)[color = gray] circle (2pt);
      \filldraw (3,1)[color = gray] circle (2pt);
      \filldraw (3,2)[color = gray] circle (2pt);
      \filldraw (4,0)[color = gray] circle (2pt);
      \filldraw (4,1)[color = gray] circle (2pt);
      \filldraw (5,0)[color = gray] circle (2pt);
      \draw (0,0) -- (1,1) -- (2,3) -- (0,5) -- (0,0);
      \draw (2.2,4) node {$P_1$};
      \draw (2,3.1) node[anchor = north west] {$\nu$};
    \end{scope}
    \end{tikzpicture}
  \end{center}

  For any vertex $\eta$ of $P_3$ and $P_1$ we have $l_\mu \geq r_\mu$, so the
  recession fan of $C$ is relatively realizable (Proposition
  \ref{prop:posreal}). However, we have $\frac{n_\mu}{s_\mu} = 1 > \frac 23 =
  \frac{s_\nu}{n_\nu}$, so $ I = \Big[ \frac{n_\mu}{s_\mu},
  \frac{s_\nu}{n_\nu} \Big] $ is empty. By Proposition \ref{prop:length},
  there are thus no $q,q' \in \mathbb R_{>0}$ such that $C$ is relatively
  realizable.
\end{Example}

 % General Criteria for Relative Realizability

  \bibliographystyle{amsalpha}
  \bibliography{references}
\end{document}